\documentclass[11pt]{article}
\usepackage{amsthm, amsmath, amssymb, amsfonts, url, booktabs, tikz, setspace, fancyhdr, bm}
\usepackage{hyperref}
\usepackage{geometry}
\usepackage{enumerate}
\usepackage[shortlabels]{enumitem}
\usepackage[babel]{microtype}
\usepackage[english]{babel}
\usepackage{comment}
\usepackage{bbm}
\usepackage{csquotes}
\usepackage{mathabx}
\usetikzlibrary{calc, angles, quotes, intersections}
\usepackage{graphicx}
\usepackage{float}
\usepackage{xcolor}

\counterwithin{figure}{section}

\newtheorem{theorem}{Theorem}[section]

\newtheorem{lemma}[theorem]{Lemma}

\newtheorem{claim}[theorem]{Claim}

\theoremstyle{definition}

\newtheorem*{defn-non}{Definition}

\usepackage[capitalise]{cleveref}

\newlist{Case}{enumerate}{2}
\setlist[Case, 1]{%
    label           =   {\bfseries Case \arabic*.},
    labelindent=1em ,labelwidth=1.3cm, labelsep*=1em, leftmargin =!
}
\setlist[Case, 2]{%
    label           =   {\bfseries Subcase \arabic{Casei}.\arabic*.},
    labelindent=-1em ,labelwidth=1.3cm, labelsep*=1em, leftmargin =!
}

\newenvironment{poc}{\begin{proof}[Proof of the claim]}{\end{proof}}

\newcommand*{\abs}[1]{\lvert#1\rvert}

\newcommand{\R}{\mathbb{R}}

\newcommand{\conv}{\operatorname{conv}}
\newcommand{\VC}{\operatorname{VC}}

\title{Tverberg's theorem for unions of convex sets: Sharp bounds and colored extensions}

\author{
Gennian Ge\thanks{School of Mathematical Sciences, Capital Normal University, Beijing 100048, China. Email: gnge@zju.edu.cn. Gennian Ge is supported by the National Key Research and Development Program of China under Grant 2025YFC3409900, the National Natural Science Foundation of China under Grant 12231014, and Beijing Scholars Program.}
\and 
Yang Shu\thanks{School of Mathematical Sciences, University of Science and Technology of China, Hefei, 230026, China. Email: shuyangyyyy@mail.ustc.edu.cn.}
\and
Zixiang Xu\thanks{School of Mathematical Sciences, Zhejiang University, Hangzhou, China. Email: zixiangxu@zju.edu.cn.}
}

\date{}

\begin{document}
\maketitle
\begin{abstract}
Let $f_{r}(d,s_{1},\ldots,s_{r})$ be the least $N$ such that every $N$-point set $P\subseteq\R^{d}$ has an $r$-partition $P=P_{1}\sqcup\cdots\sqcup P_{r}$ with the following property: whenever $C_{i}\supseteq P_{i}$ is a union of at most $s_{i}$ convex sets, one has $\bigcap_{i=1}^{r}C_{i}\ne\emptyset$. A recent breakthrough of Alon and Smorodinsky proved that $f_{r}(d,s,\ldots,s)\le cdr^{2}s^{r}\log r\log(es^{r})$ for an absolute constant $c>0$. In this paper, we determine the asymptotic order in two principal ranges: $f_{2}(2,s,s)=\Theta(s^{2})$, and $f_{r}(d,s,\ldots,s)=\Theta_{d,r}(s^{r}\log{s})$ for every fixed $d,r$ with $d\ge r+2$. The first one determines the order of the extremal function proposed by Kalai from the 1970s. Together, the two results show a sharp dependence on the dimension: for two parts, the logarithmic factor disappears in the plane but is necessary in every fixed dimension $d\ge4$.

Beyond these sharp results, when $r\ge d+1$ we improve the upper bound of Alon and Smorodinsky by proving both $f_{r}(d,s,\ldots,s)\le c_{d}rs^{r}\log(ers^{r})$ and $f_{r}(d,s,\ldots,s)\le c_{d}r^{d+2}s^{d+1}\log(ers)$ through a local Helly-type argument. We also prove $f_{r}(d,s,\ldots,s)>s^{\min\{r,d\}}$ for every $d\ge2$. Finally, we study two colored analogues. The direct B\'{a}r\'{a}ny--Larman-type extension, in which one seeks $r$ disjoint rainbow sets chosen from $d+1$ color classes, fails as soon as two convex pieces are allowed. Nevertheless, a different extension does hold: given sufficiently many prescribed $r$-point classes, one can split every class completely among the $r$ final parts while retaining the required intersection property.
\end{abstract}

\section{Introduction}
Radon's theorem \cite{Radon1921}, one of the basic results of discrete geometry, states that every \(d+2\) points in \(\R^{d}\) can be split into two parts whose convex hulls intersect. Tverberg's theorem \cite{Tverberg1966} is the \(r\)-part extension: every \((r-1)(d+1)+1\) points in \(\R^{d}\) admit a partition into \(r\) parts whose convex hulls have a common point. The theorem has many topological and combinatorial extensions, see, for example, \cite{BjornerLovaszVrecicaZivaljevic1994,BlagojevicMatschkeZiegler2015,Matousek1996,MatousekTancerWagner2012,ZivaljevicVrecica1992} and the survey \cite{BaranySoberon2018}. For the present problem, it is useful to restate Tverberg's conclusion as follows. A partition \(P=P_{1}\sqcup\cdots\sqcup P_{r}\) satisfies \(\bigcap_{i=1}^{r}\operatorname{conv}(P_{i})\ne\emptyset\) if and only if every choice of convex sets \(C_{i}\supseteq P_{i}\), \(i\in [r]\), has nonempty total intersection.

This restatement suggests replacing each convex set \(C_{i}\) by a controlled union of convex sets. We call a set \(s\)-convex if it is a union of at most \(s\) convex sets, so the integer \(s\) measures how far the set is allowed to be from convex. Let \(r\ge2\), \(d\ge1\), and \(s_{1},\ldots,s_{r}\ge1\). Following \cite{AlonSmorodinsky2025}, let \(f_{r}(d,s_{1},\ldots,s_{r})\) be the least integer \(N\) with the following property. Whenever \(P\subseteq\R^{d}\) is finite and \(\abs{P}\ge N\), there is a partition \(P=P_{1}\sqcup\cdots\sqcup P_{r}\) into nonempty parts such that every choice of \(s_{i}\)-convex sets \(C_{i}\supseteq P_{i}\), \(i\in [r]\), satisfies \(\bigcap_{i=1}^{r}C_{i}\ne\emptyset\). Equivalently, we seek a surjective map \(\chi:P\to [r]\), with \(P_{i}=\chi^{-1}(i)\), whose parts cannot be covered by \(s_{i}\)-convex sets with empty total intersection. Radon's theorem gives \(f_{2}(d,1,1)=d+2\), while Tverberg's theorem gives \(f_{r}(d,1,\ldots,1)=(r-1)(d+1)+1\).

Thus the new difficulty begins as soon as some \(s_{i}>1\). An \(s_{i}\)-convex set \(C_{i}\) may contain \(P_{i}\) without containing \(\conv(P_{i})\), so the common point supplied by Tverberg's theorem need not lie in all the sets \(C_{i}\). We ask how large \(P\) must be before one can choose a partition for which every such collection of covers still has a common point.

The problem originates in a question of Kalai from the 1970s about a Radon theorem for unions of convex sets, see \cite{Kalai2017,2025KalaiBolg}. A first finiteness argument, due to B\'{a}r\'{a}ny and Kalai, shows that \(f_{2}(d,s_{1},s_{2})\) is finite for all \(d,s_{1},s_{2}\), but its Ramsey-type bound is enormous. This argument is related to work of B\'{a}r\'{a}ny, Matou\v{s}ek, and P\'{o}r \cite{BaranyMatousekPor2016} and to the semi-algebraic Ramsey theorem of Conlon, Fox, Pach, Sudakov, and Suk \cite{TRansAMS2014}. A recent breakthrough of Alon and Smorodinsky \cite{AlonSmorodinsky2025} established, for an absolute constant \(c>0\), the first effective general estimate
\[
f_{r}(d,s_{1},\ldots,s_{r})\le cdr^{2}\log r\left(\prod_{i=1}^{r}s_{i}\right)\log\left(e\prod_{i=1}^{r}s_{i}\right).
\]
For equal parameters this is \(O(dr^{2}s^{r}\log r\log(es^{r}))\). In the original planar two-part setting it gives \(f_{2}(2,s,s)=O(s^{2}\log(s+1))\). On the other hand, Chen, Ge, Shu, Wang, and Xu \cite{ChenGeShuWangXu2025} constructed \(s^{2}\) planar points for which every red--blue partition is bad, and therefore proved \(f_{2}(2,s,s)>s^{2}\). The logarithmic gap between these bounds was the first unresolved case.

Our main result closes this gap completely.

\begin{theorem}\label{thm:planar}
There exist positive absolute constants \(c_{1},c_{2}\) such that, for every positive integer \(s\),
\[
c_{1}s^{2}\le f_{2}(2,s,s)\le c_{2}s^{2}.
\]
\end{theorem}
The lower bound was already shown in~\cite{ChenGeShuWangXu2025} with \(c_{1}=2\), and the upper bound in
Theorem~\ref{thm:planar} is one of the main results of this paper and determines the asymptotic order in the planar case of Kalai's original two-part problem. A particularly surprising feature is the role of the dimension. In the two-part problem, the logarithmic factor in the Alon--Smorodinsky bound disappears completely when \(d=2\). In contrast, our next theorem shows that the same factor is unavoidable in every fixed dimension \(d\ge4\): for \(r=2\), it gives \(f_{2}(d,s,s)=\Theta_{d}(s^{2}\log s)\). Thus the planar case is not merely the first instance of the higher dimensional result, it has genuinely smaller order of growth.

More generally, once \(d\ge r+2\), the logarithmic factor is necessary for every fixed number of parts.

\begin{theorem}\label{thm:lower}
There is an absolute constant \(c>0\) such that, for all \(r\ge2\), \(s\ge1\), and \(d\ge r+2\),
\[
f_{r}(d,s,\ldots,s)\ge c(d-r+2)s^{r}\log(s+1).
\]
\end{theorem}

Theorem \ref{thm:lower} and the upper bound of Alon and Smorodinsky~\cite{AlonSmorodinsky2025} determine the dependence on \(s\) throughout this range. For every fixed \(d\) and \(r\) with \(d\ge r+2\),
\[
f_{r}(d,s,\ldots,s)=\Theta_{d,r}(s^{r}\log{s}).
\]

An adaptation of the same construction also improves the range of the previously known power lower bound.

\begin{theorem}\label{thm:power-lower}
For all \(r\ge2\), \(s\ge1\), and \(d\ge2\),
\[
f_{r}(d,s,\ldots,s)>s^{\min\{r,d\}}.
\]
\end{theorem}

Thus \(f_{r}(d,s,\ldots,s)>s^{r}\) already holds for \(d\ge r\), improving the earlier condition \(d\ge2r-2\). In the complementary range \(r> d\), Theorem \ref{thm:power-lower} gives the lower bound \(s^{d}\). This points to a concrete remaining gap: our new upper bound in that range is of order \(s^{d+1}\log{s}\), whereas the new construction gives \(s^{d}\).

We next state our upper bound. Write \(\boldsymbol{s}=(s_{1},\ldots,s_{r})\), \(S(\boldsymbol{s})=\prod_{i=1}^{r}s_{i}\), and \(q=\min\{r,d+1\}\). Throughout this paper, we put
\[
H_{d,r}(\boldsymbol{s})=\sum_{\substack{\emptyset\ne I\subseteq [r]\\ \abs{I}\le d+1}}\abs{I}\prod_{i\in I}s_{i}
\]
and \(T_{d,r}(\boldsymbol{s})=\min\{qS(\boldsymbol{s}),H_{d,r}(\boldsymbol{s})\}\). The sum defining \(H_{d,r}(\boldsymbol{s})\) runs over choices of at most \(d+1\) parts. For each such choice \(I\), the product counts the possible choices of one convex component from every part in \(I\), and the factor \(\abs{I}\) records how many sets are involved.
\begin{theorem}\label{thm:general-upper}
There is an absolute constant \(c\) such that, for all \(r\ge 2\), \(d\ge 1\), and \(s_{1},\ldots,s_{r}\ge 1\),
\[
f_{r}(d,s_{1},\ldots,s_{r})\le cdrT_{d,r}(\boldsymbol{s})\log(erT_{d,r}(\boldsymbol{s})).
\]
\end{theorem}

The key ingredient is a local Helly-type analysis. Instead of counting information attached to all \(r\) parts at once, we first look at each local obstruction and keep only the part indices needed to witness the empty intersection.

Let us make the equal-parameter consequences explicit. If \(s_{1}=\cdots=s_{r}=s\), then
\[
H_{d,r}(s,\ldots,s)=\sum_{k=1}^{q}k\binom{r}{k}s^{k}.
\]
When \(r\le d+1\), one has \(T_{d,r}(s,\ldots,s)=rs^{r}\), and Theorem~\ref{thm:general-upper} gives \(f_{r}(d,s,\ldots,s)\le cdr^{2}s^{r}\log(er^{2}s^{r})\). For large \(s\), the visible improvement is essentially the missing \(\log r\) factor. Here and below, \(c_{d}\) denotes a positive constant depending only on \(d\), whose value may change from one occurrence to the next. When \(r\ge d+1\), one has
\[
T_{d,r}(s,\ldots,s)\le \min\{(d+1)s^{r},c_{d}r^{d+1}s^{d+1}\}.
\]
The first term gives \(f_{r}(d,s,\ldots,s)\le c_{d}rs^{r}\log(ers^{r})\). The second term gives
\[
f_{r}(d,s,\ldots,s)\le c_{d}r^{d+2}s^{d+1}\log(ers).
\]
The second upper bound is stronger at the level of the main factors whenever \(s^{r-d-1}\ge c_{d}r^{d+1}\). In particular, when \(s\gg r\gg d\), it significantly improves the Alon--Smorodinsky bound.

For reference, the next three tables collect and compare the known bounds for equal parameters. The sharp planar estimate is listed first. The ranges in the upper-bound table may overlap, since both estimates from Theorem~\ref{thm:general-upper} can be useful when \(r\ge d+1\). The remaining rows record the broader parameter ranges, and the third table places the best current bounds side by side when \(d\) and \(r\) are fixed and \(s\) tends to infinity.

\begin{center}
\small
\textbf{Upper bounds for \(f_{r}(d,s,\ldots,s)\)}\par\smallskip
\renewcommand{\arraystretch}{1.18}
\begin{tabular}{@{}p{0.18\textwidth}p{0.25\textwidth}p{0.49\textwidth}@{}}
\toprule
Parameter range & Position & Upper bound \\
\midrule
\(d=r=2\) & Theorem~\ref{thm:planar} & \(O(s^{2})\) \\
\(d\ge1\), \(r\ge2\) & Alon--Smorodinsky \cite{AlonSmorodinsky2025} & \(O\bigl(dr^{2}s^{r}\log r\log(es^{r})\bigr)\) \\
\(2\le r\le d+1\) & Theorem~\ref{thm:general-upper} & \(O\bigl(dr^{2}s^{r}\log(er^{2}s^{r})\bigr)\) \\
\(r\ge d+1\) & Theorem~\ref{thm:general-upper} & \(O_{d}\bigl(rs^{r}\log(ers^{r})\bigr)\) \\
\(r\ge d+1\) & Theorem~\ref{thm:general-upper} & \(O_{d}\bigl(r^{d+2}s^{d+1}\log(ers)\bigr)\) \\
\bottomrule
\end{tabular}
\end{center}

\begin{center}
\small
\textbf{Lower bounds for \(f_{r}(d,s,\ldots,s)\)}\par\smallskip
\renewcommand{\arraystretch}{1.18}
\begin{tabular}{@{}p{0.18\textwidth}p{0.25\textwidth}p{0.49\textwidth}@{}}
\toprule
Parameter range & Position & Lower bound \\
\midrule
\(d=r=2\) & Chen-Ge-Shu-Wang-Xu~\cite{ChenGeShuWangXu2025} & \(f_{2}(2,s,s)>s^{2}\) \\
\(d\ge2r-2\) & Chen-Ge-Shu-Wang-Xu~\cite{ChenGeShuWangXu2025} & \(f_{r}(d,s,\ldots,s)>(d-2r+4)s^{r}\) \\
\(2\le d<r\) & Theorem~\ref{thm:power-lower} & \(f_{r}(d,s,\ldots,s)>s^{d}\) \\
\(r\le d\le r+1\) & Theorem~\ref{thm:power-lower} & \(f_{r}(d,s,\ldots,s)>s^{r}\) \\
\(d\ge r+2\) & Theorem~\ref{thm:lower} & \(f_{r}(d,s,\ldots,s)\ge c(d-r+2)s^{r}\log(s+1)\) \\
\bottomrule
\end{tabular}
\end{center}

\begin{center}
\small
\textbf{Comparison for fixed \(d,r\) as \(s\to\infty\)}\par\smallskip
\renewcommand{\arraystretch}{1.18}
\begin{tabular}{@{}p{0.14\textwidth}p{0.23\textwidth}p{0.24\textwidth}p{0.30\textwidth}@{}}
\toprule
Range & Best lower bound & Best upper bound & Current status \\
\midrule
\(d=r=2\) & \(\Omega(s^{2})\) & \(O(s^{2})\) & Asymptotically sharp \\
\(d\ge r+2\) & \(\Omega_{d,r}\bigl(s^{r}\log{s}\bigr)\) & \(O_{d,r}\bigl(s^{r}\log{s}\bigr)\) & Asymptotically sharp \\
\(r\le d\le r+1\) & \(\Omega_{d,r}(s^{r})\) & \(O_{d,r}\bigl(s^{r}\log{s}\bigr)\) & Sharp up to \(\log{s}\) \\
\(2\le d<r\) & \(\Omega_{d,r}(s^{d})\) & \(O_{d,r}\bigl(s^{d+1}\log{s}\bigr)\) & Gap at most \(s\log{s}\) \\
\bottomrule
\end{tabular}
\end{center}

Taken together, the tables isolate both the solved ranges and the remaining gaps in the dependence on \(s\). The original planar case \(d=r=2\) is now sharp, as is every fixed pair \((d,r)\) with \(d\ge r+2\). Apart from the planar case, when \(d\in\{r,r+1\}\) only a logarithmic gap remains, while for \(2\le d<r\) the best bounds differ by one power of \(s\), up to a logarithmic factor.

We next turn to colored versions. In the classical colored extensions of Tverberg's theorem, the partition must respect prescribed classes of points. The theorem of B\'ar\'any and Larman \cite{BaranyLarman1992}, for example, asks for pairwise disjoint rainbow sets, each taking one point from each prescribed class, whose convex hulls meet. This naturally raises the question whether the B\'ar\'any--Larman rainbow formulation remains valid after replacing the convex-hull intersection condition by the robust requirement that every choice of \(s\)-convex supersets of the selected rainbow sets has a common point. The answer is negative: for unions of convex sets, the direct rainbow analogue fails in general.

Formally, let \(\gamma_{s}(r,d)\) be the least integer \(t\) with the following property. For every collection of pairwise disjoint finite sets \(Y_{0},\ldots,Y_{d}\subseteq\R^{d}\) with \(\abs{Y_{j}}\ge t\), there are pairwise disjoint rainbow sets \(R_{1},\ldots,R_{r}\), meaning \(\abs{R_{i}\cap Y_{j}}=1\) for every \(i\in [r]\) and \(j\in\{0,\ldots,d\}\), such that for every choice of \(s\)-convex sets \(C_{i}\supseteq R_{i}\), \(i\in [r]\), one has \(\bigcap_{i=1}^{r}C_{i}\ne\emptyset\). If no such \(t\) exists, set \(\gamma_{s}(r,d)=\infty\). Our first colored result is the following.

\begin{theorem}\label{thm:rainbow-obstruction}
For all \(r\ge2\), \(d\ge1\), and \(s\ge2\), one has \(\gamma_{s}(r,d)=\infty\).
\end{theorem}

The obstruction already uses two convex pieces. On the moment curve, any two disjoint \((d+1)\)-point sets can each be split into two parts so that the convex hull of any part from the first set is disjoint from the convex hull of any part from the second. Thus the desired robust intersection fails before any counting argument becomes relevant.

Furthermore, we find that the correct positive statement should keep more of the prescribed classes. Instead of selecting only \(d+1\) points for each final part, we start with many pairwise disjoint \(r\)-point classes \(A_{1},\ldots,A_{m}\), and split every class completely among the \(r\) final parts. In other words, each final part receives exactly one point from every prescribed class.

Let \(g_{r}(d,s_{1},\ldots,s_{r})\) be the least \(m\) with the following property. Whenever \(A_{1},\ldots,A_{m}\subseteq\R^{d}\) are pairwise disjoint sets of size \(r\), there is a labeled partition \(A_{1}\sqcup\cdots\sqcup A_{m}=P_{1}\sqcup\cdots\sqcup P_{r}\) such that \(\abs{P_{i}\cap A_{j}}=1\) for every \(i\in [r]\) and \(j\in [m]\), and such that every choice of \(s_{i}\)-convex sets \(C_{i}\supseteq P_{i}\) has nonempty total intersection.

\begin{theorem}\label{thm:block}
There is an absolute constant \(c\) such that, for all \(r\ge 2\), \(d\ge 1\), and \(s_{1},\ldots,s_{r}\ge 1\),
\[
g_{r}(d,s_{1},\ldots,s_{r})\le cdT_{d,r}(\boldsymbol{s})\log(erT_{d,r}(\boldsymbol{s})).
\]
\end{theorem}
\medskip
\noindent\textbf{Note added:} The conference version of this draft was submitted in early July 2026. While preparing the journal version, we learned of independent work by Keller and Smorodinsky~\cite{2026coloredVC}, which proves Theorem~\ref{thm:block} from a different point of view. Their proof introduces a colored form of VC-dimension, measuring rainbow partitions of prescribed classes, and bounds it in terms of ordinary VC-dimension. Their result gives a slightly weaker bound, but it holds in separable abstract convexity spaces with given Radon number, and hence in a substantially broader setting than our result in \(\R^{d}\).

\section{Some auxiliary results}\label{sec:auxiliary}

For ease of reference, we collect here the results from the literature that will be used as black boxes. We group them by their role and explain some basics.

\subsection{Planar ingredients}

We first introduce some geometric results for the planar counting argument. A line \(\ell\) is a \emph{supporting line} of a compact convex set \(K\subseteq\R^{2}\) if \(K\) lies in one of the closed halfplanes bounded by \(\ell\) and \(K\cap\ell\ne\emptyset\). Suppose that \(\ell\) supports two disjoint convex polygons \(K_{i}\) and \(K_{j}\) at vertices \(x\) and \(y\), respectively. We call \([x,y]\) a \emph{free supporting segment} if the open segment \([x,y]\setminus\{x,y\}\) meets none of the polygons under consideration. Such segments are \emph{noncrossing} if their open segments are pairwise disjoint. A free supporting segment is \emph{interior} if its supporting line separates its endpoint polygons, and \emph{exterior} otherwise. A noncrossing family is \emph{pointed} if, after adding the polygon boundaries, every vertex is incident with an angle larger than \(\pi\). A maximal pointed noncrossing family of free supporting segments is a \emph{pseudo-triangulation}, its bounded free faces are \emph{pseudotriangles}, whose three corners are called cusp points.

Given a pseudo-triangulation \(\mathcal{E}\), first extend each interior segment within the two pseudotriangles having it on their boundary, as in \cite[Section~3]{PocchiolaVegter1999}. The resulting graph is the \emph{screen graph} of \(\mathcal{E}\). A vertex incident with only one edge has degree one. Next choose an order of the degree-one vertices and treat them one at a time. At each such vertex, continue its incident edge beyond the vertex along the same line until it first meets the current graph, or to infinity if there is no such point. The resulting extended screen graph determines a \emph{supporting-line subdivision} of \(\mathcal{E}\), its cells are the closures of the connected components of its complement.

For polygons, we use throughout the rounding convention described in \cite[Section~9.2]{RoteSantosStreinu2008}. Degree-one vertices that correspond to the same point are kept distinct, each is identified by the interior segment and pseudotriangle from which it arises.

We say that the polygons are in \emph{supporting-line general position} if every common supporting line touches each polygon at a unique vertex, no line supports three polygons, and no three common supporting lines pass through the same point outside the polygons. The next lemma is the geometric fact that we need.

\begin{lemma}[\cite{PocchiolaVegter1999,RoteSantosStreinu2008}]\label{lem:supporting-subdivision}
Let \(K_{1},\ldots,K_{k}\), where \(k\ge3\), be pairwise disjoint convex polygons with nonempty interiors in supporting-line general position. There is a noncrossing family \(\mathcal{E}\) of \(3k-3\) free supporting segments forming a pseudo-triangulation whose endpoint set meets every \(K_i\). Every supporting-line subdivision of \(\mathcal{E}\) has convex cells, with \(K_{1},\ldots,K_{k}\) in distinct cells.
\end{lemma}

The screen graph is fixed by \(\mathcal{E}\), but the second step may give different supporting-line subdivisions. The last assertion of the lemma holds for every order chosen in that step.

A \emph{plane straight-line graph} is a graph drawn with straight edges so that two edges meet only at a common endpoint. The segments in \(\mathcal{E}\) form such a graph. For each segment, we will also record which side of its supporting line contains the polygon at each endpoint.

\begin{lemma}[\cite{RoteSantosStreinu2008}]\label{lem:plane-graph-count}
There is an absolute constant \(c_{\mathrm{pg}}>0\) such that, for every \(m\)-point set \(S\subseteq\R^{2}\) with no three collinear points, the number of plane straight-line graphs with vertex set contained in \(S\) is at most \(c_{\mathrm{pg}}^{m}\).
\end{lemma}

A \emph{fence} is the common polygonal boundary of two closed regions with disjoint interiors that together cover the plane. We say that it separates disjoint sets \(R\) and \(B\) if one region contains \(R\) and the other contains \(B\). The length of a fence is the sum of the lengths of its edges. The following existence theorem is the only external result used in our fence argument.

\begin{lemma}[\cite{AbrahamsenGiannopoulosLofflerRote2020}]\label{lem:shortest-fence}
Let \(R\) and \(B\) be disjoint compact polygonal sets in the plane. Among all fences that place \(R\) and \(B\) in different regions, there is one of minimum length.
\end{lemma}

The lower bound in Theorem~\ref{thm:planar}, already mentioned in the introduction, follows from the next construction. We state its stronger form because it will also start the proof of Theorem~\ref{thm:power-lower}.

\begin{lemma}[\cite{ChenGeShuWangXu2025}]\label{lem:planar-bad-set}
For every \(s\ge1\), there is a set \(X\subseteq\R^{2}\) of \(s^{2}\) points with the following property. For every map \(\chi:X\to[2]\), each set \(\chi^{-1}(i)\) can be written as \(X_{i,1}\sqcup\cdots\sqcup X_{i,s}\) so that
\[
\conv(X_{1,a})\cap\conv(X_{2,b})=\emptyset
\]
for all \(a,b\in[s]\).
\end{lemma}

\subsection{Convexity and counting tools}

For the rest of the paper, a \emph{closed halfspace} means either \(\R^{d}\) or a set of the form \(\{\boldsymbol{x}\in\R^{d}:\langle\boldsymbol{u},\boldsymbol{x}\rangle\le a\}\) with \(\boldsymbol{u}\ne\boldsymbol{0}\). A \emph{convex polyhedron} is an intersection of finitely many closed halfspaces, where the intersection of an empty family of halfspaces is interpreted as \(\R^{d}\).

We shall repeatedly separate compact convex sets from closed convex sets.

\begin{lemma}[\cite{Rockafellar1970}]\label{lem:strict-separation}
Let \(K\) be a compact convex set and \(L\) be a closed convex set in a finite-dimensional Euclidean space. If \(K\cap L=\emptyset\), then some affine hyperplane strictly separates \(K\) and \(L\).
\end{lemma}

We use Helly's theorem in its usual finite form.

\begin{lemma}[\cite{Helly1923}]\label{lem:helly}
If a finite family of convex sets in \(\R^{d}\) has empty intersection, then some subfamily of size at most \(d+1\) has empty intersection.
\end{lemma}

We also use the following well-known simultaneous separation statement. It is a standard consequence of Helly's theorem, we take advantage of the precise form needed below and refer to Gale and Klee~\cite{GaleKlee1959}.

\begin{lemma}[\cite{GaleKlee1959}]\label{lem:separation}
Let \(D_{1},\ldots,D_{t}\) be nonempty compact convex sets in \(\R^{d}\) with \(\bigcap_{i=1}^{t}D_{i}=\emptyset\). Then there are closed halfspaces \(H_{i}\supseteq D_{i}\), \(i\in [t]\), such that \(\bigcap_{i=1}^{t}H_{i}=\emptyset\).
\end{lemma}

The VC-dimension argument later in the paper uses Radon's theorem and the Sauer--Shelah lemma in the following forms. In particular, Radon's theorem implies that the traces of closed halfspaces on a finite set have VC-dimension at most \(d+1\): a shattered set of \(d+2\) points would allow one class of a Radon partition to be separated from the other, contradicting the intersection of their convex hulls.

\begin{lemma}[\cite{Radon1921}]\label{lem:radon}
Every set of \(d+2\) points in \(\R^{d}\) can be partitioned into two nonempty parts whose convex hulls intersect.
\end{lemma}

Let \(\mathcal{F}\) be a family of subsets of a set \(X\). A subset \(Y\subseteq X\) is shattered by \(\mathcal{F}\) if for every \(Z\subseteq Y\) there is \(F\in\mathcal{F}\) such that \(Z=Y\cap F\). The VC-dimension \(\VC(\mathcal{F})\) is the largest size of a shattered subset of \(X\).

\begin{lemma}[\cite{Sauer1972,Shelah1972,1971TPAVCVC}]\label{lem:sauer}
Let \(\mathcal{F}\) be a family of subsets of an \(m\)-element set \(X\). If \(\VC(\mathcal{F})\le k\), then \(\abs{\mathcal{F}}\le\sum_{j=0}^{k}\binom{m}{j}\).
\end{lemma}

The proof of Theorem~\ref{thm:lower} starts from the sharp construction for unions of halfspaces due to Csik\'{o}s, Mustafa, and Kupavskii~\cite{2019JMLR}, in a form where every normal vector has positive coordinates. For \(s=1\), take \(X=\{\boldsymbol{e}_{1},\ldots,\boldsymbol{e}_{\ell}\}\): given \(Y\subseteq X\), set the \(i\)-th coordinate of \(\boldsymbol{u}\) equal to \(1\) when \(\boldsymbol{e}_{i}\in Y\) and to \(2\) otherwise, and use the halfspace \(\{\boldsymbol{x}:\langle\boldsymbol{u},\boldsymbol{x}\rangle\le\frac{3}{2}\}\). After decreasing the absolute constant, the statement holds uniformly for all \(s\ge1\).

\begin{lemma}[\cite{2019JMLR}]\label{lem:polytope-vc}
There is an absolute constant \(c_{0}>0\) such that, for all \(\ell\ge4\) and \(s\ge1\), there is a finite set \(X\subseteq\R^{\ell}\) with \(\abs{X}\ge c_{0}\ell s\log(s+1)\) having the following property. For every \(Y\subseteq X\), there are closed halfspaces \(h_{1},\ldots,h_{p}\), where \(1\le p\le s\), such that \(Y=X\cap\bigcup_{a=1}^{p}h_{a}\), and each \(h_{a}\) can be written as
\[
h_{a}=\{\boldsymbol{x}\in\R^{\ell}:\langle\boldsymbol{u}_{a},\boldsymbol{x}\rangle\le c_{a}\}
\]
with \(c_{a}\in\R\) and every coordinate of \(\boldsymbol{u}_{a}\) strictly positive.
\end{lemma}

\subsection{An elementary component reduction}

For positive integers \(s_{1},\ldots,s_{r}\), a \emph{full component choice} is a tuple \(\boldsymbol{a}=(a_{1},\ldots,a_{r})\in\prod_{i=1}^{r}[s_{i}]\), namely one component index for every part. The following elementary reduction will be used both in the planar argument and in the general upper bound. It replaces arbitrary convex components by compact hulls of assigned points while preserving every obstruction.

\begin{lemma}\label{lem:component-reduction}
Let \(P_{1},\ldots,P_{r}\) be nonempty finite subsets of \(\R^{d}\). Suppose there are \(s_{i}\)-convex sets \(C_{i}\supseteq P_{i}\) with \(\bigcap_{i=1}^{r}C_{i}=\emptyset\). Then, for every \(i\in [r]\) and \(a\in [s_{i}]\), there is a nonempty compact convex set \(D_{i,a}\) such that \(P_{i}\subseteq\bigcup_{a=1}^{s_{i}}D_{i,a}\), and for every full component choice \(\boldsymbol{a}\) one has \(\bigcap_{i=1}^{r}D_{i,a_{i}}=\emptyset\).
\end{lemma}
\begin{proof}[Proof of Lemma~\ref{lem:component-reduction}]
Write \(C_{i}=\bigcup_{a=1}^{s_{i}}C_{i,a}\), repeating one component if fewer than \(s_{i}\) are used. For each point of \(P_{i}\), choose one component containing it. If a component receives no point, repeat one that does. We let \(D_{i,a}\) be the convex hull of the points assigned to the resulting component \(a\). These sets are nonempty, compact, and convex, and their union contains \(P_{i}\).

Fix a full component choice \(\boldsymbol{a}\). Each \(D_{i,a_{i}}\) lies in a convex component of \(C_{i}\). Hence a point in \(\bigcap_{i=1}^{r}D_{i,a_{i}}\) would lie in every \(C_{i}\), contrary to \(\bigcap_{i=1}^{r}C_{i}=\emptyset\).
\end{proof}

\subsection{A permanent bound}

For an \(r\times r\) matrix \(M=(m_{i,j})\), its permanent is
\[
\operatorname{per}(M)=\sum_{\sigma\in S_{r}}\prod_{i=1}^{r}m_{i,\sigma(i)},
\]
where \(S_{r}\) is the symmetric group on \([r]\). For a zero-one matrix, this is the number of permutations \(\sigma\in S_{r}\) such that \(m_{i,\sigma(i)}=1\) for every \(i\in[r]\). We use Bregman's bound in the following form.

\begin{lemma}[\cite{Bregman1973}]\label{lem:bregman}
Let \(M\) be an \(r\times r\) zero-one matrix with row sums \(a_{1},\ldots,a_{r}\). If all \(a_{j}\) are positive, then \(\operatorname{per}(M)\le\prod_{j=1}^{r}(a_{j}!)^{\frac{1}{a_{j}}}\). If some \(a_{j}=0\), then \(\operatorname{per}(M)=0\).
\end{lemma}

\section{The sharp planar case}\label{sec:planar}

\subsection{Overview of the proof}

In view of the lower bound recalled in the introduction, it suffices to prove \(f_{2}(2,s,s)=O(s^{2})\). Let \(P\subseteq\R^{2}\) be finite. A red--blue coloring \(P=P_{R}\sqcup P_{B}\) is \emph{bad} if there are convex sets \(K_{1},\ldots,K_{p}\) and \(L_{1},\ldots,L_{q}\), where \(p,q\le s\), such that \(P_{R}\subseteq\bigcup_{i=1}^{p}K_{i}\), \(P_{B}\subseteq\bigcup_{j=1}^{q}L_{j}\), and \(\left(\bigcup_{i=1}^{p}K_{i}\right)\cap\left(\bigcup_{j=1}^{q}L_{j}\right)=\emptyset\). Thus we need to find a coloring that is not bad and uses both colors whenever \(\abs{P}=cs^{2}\) for some constant \(c>0\).

A partition \(P=B_{1}\sqcup\cdots\sqcup B_{k}\) is \emph{crossing-free} if the convex hulls \(\conv(B_{1}),\ldots,\conv(B_{k})\) are pairwise disjoint. Our proof has three steps. First, we count crossing-free partitions. We show that every such partition into \(k\) blocks can be recovered from at most \(6k-6\) selected points and a bounded amount of additional information. Second, this bound allows us to choose a red--blue coloring that has no crossing-free monochromatic partition with only \(O(s^{2})\) blocks. Finally, we start from any bad coloring and construct exactly such a partition. For the last step, we separate each red convex set from each blue one and use the separating lines to obtain polygons with \(O(s^{2})\) vertices in total. We then cut the plane along a shortest fence between the two colors and triangulate the resulting regions.

The second and third steps contradict each other once \(\abs{P}=cs^{2}\) for some sufficiently large constant \(c>0.\) Next, we divide the proof into three main steps and state them in full details.

\subsection{Counting crossing-free partitions}

Let \(N_{k}(P)\) denote the number of unordered crossing-free \(k\)-partitions of \(P\). We first provide an upper bound as follows.

\begin{lemma}\label{lem:planar-count}
There is an absolute constant \(c_{\mathrm{pc}}\ge1\) such that every \(n\)-point set \(P\subseteq\R^{2}\) and every \(1\le k\le n\) satisfy
\[
N_{k}(P)\le\left(\frac{c_{\mathrm{pc}}n}{k}\right)^{6k}.
\]
\end{lemma}

\begin{proof}[Proof of Lemma~\ref{lem:planar-count}]
We first label the points of \(P\). We call a labeled point set \emph{regular} if no three points are collinear, no two lines determined by pairs of points are parallel, and no three such lines meet outside the point set. We may assume that \(P\) is regular. To see this, consider all crossing-free partitions of \(P\). There are only finitely many, and the distance between the convex hulls of any two blocks in any one of them is positive. We can therefore move the points by a sufficiently small amount without destroying any of these partitions. At the same time, we choose the perturbation so that no three points become collinear, no two determined lines become parallel, and no three determined lines meet outside the point set. The perturbed set may have additional crossing-free partitions, but it retains every partition of the original set, which is all we need for an upper bound.

Fix a crossing-free partition \(\Pi=\{B_{1},\ldots,B_{k}\}\) with \(k\ge3\), and put \(K_{i}=\conv(B_{i})\) for \(i\in[k]\). Call a block \emph{large} if it has at least three points and \emph{small} otherwise. Let \(I=\{i\in[k]:\abs{B_{i}}\ge3\}\) and \(k'=\abs{I}\). Since \(P\) is regular, the sets \(K_{i}\), \(i\in I\), are convex polygons with nonempty interiors. They are pairwise disjoint and are in the supporting-line general position defined before Lemma~\ref{lem:supporting-subdivision}.

\begin{claim}\label{claim:small-support}
If \(k'\ge3\), we can choose a family \(\mathcal{E}\) as in Lemma~\ref{lem:supporting-subdivision} whose endpoint set \(S_{1}\subseteq\bigcup_{i\in I}B_{i}\) satisfies \(\abs{S_{1}}\le6k'-6\) and meets every large block.
\end{claim}

\begin{poc}
We apply Lemma~\ref{lem:supporting-subdivision} to the polygons \(K_{i}\), \(i\in I\), and let \(\mathcal{E}\) be the resulting family. By general position, each endpoint is a vertex of one of these polygons and hence a point of the corresponding block. We take \(S_{1}\) to be the set of endpoints. Then \(\abs{S_{1}}\le2\abs{\mathcal{E}}=6k'-6\), and every large block meets \(S_{1}\).
\end{poc}

We next assign a set \(S\) and a family \(\mathcal{E}\) to the partition. Put \(S_{0}=\bigcup_{i\notin I}B_{i}\), so \(\abs{S_{0}}\le2(k-k')\). If \(k'\ge3\), choose one family supplied by Claim~\ref{claim:small-support} and let \(S_{1}\) be its endpoint set. If \(k'=2\), choose a segment \([x,y]\) whose line supports both polygons and places them in opposite closed halfplanes, set \(\mathcal{E}=\{[x,y]\}\) and \(S_{1}=\{x,y\}\). If \(k'=1\), let \(S_{1}\) consist of one point in the unique large block and put \(\mathcal{E}=\emptyset\). If \(k'=0\), put \(S_{1}=\mathcal{E}=\emptyset\). We make these choices once and for all for each partition.

Set \(S=S_{0}\cup S_{1}\). Every block meets \(S\), every small block is contained in \(S\), and \(\abs{\mathcal{E}}\le3k-3\). Moreover,
\[
k\le\abs{S}\le6k-6.
\]
Indeed, the upper bound follows from \(2(k-k')+6k'-6\) when \(k'\ge3\), and from \(2k-2\), \(2k-1\), or \(2k\) when \(k'=2,1,0\), respectively.

We next count how many partitions can select the same set \(S\).

\begin{claim}\label{claim:discrete-record}
There is an absolute constant \(c_{\mathrm{rec}}\ge1\) such that, for every fixed \(m\)-point set \(S\subseteq P\), at most \(c_{\mathrm{rec}}^{m}\) crossing-free partitions select \(S\) by the preceding rule.
\end{claim}

\begin{poc}
First, the segments in \(\mathcal{E}\) form a plane straight-line graph \(G_{\mathcal{E}}\) whose vertices lie in \(S\). Second, for every nonempty set \(B_{i}\cap S\), join its first point to all its other points. These stars, together with the isolated points of \(S\), form another plane straight-line graph \(G_{\Pi}\) on \(S\). Indeed, each star lies in \(K_{i}\), and the polygons \(K_{i}\) are pairwise disjoint. The connected components of \(G_{\Pi}\) are precisely the sets \(B_{i}\cap S\). For each component, we also record whether the corresponding block is small.

We apply Lemma~\ref{lem:plane-graph-count} to both graphs. There are at most \(c_{\mathrm{pg}}^{2m}\) ordered pairs \((G_{\mathcal{E}},G_{\Pi})\). For each segment in \(\mathcal{E}\), we also record which closed halfplane contains the polygon at each endpoint. This gives at most \(4^{\abs{\mathcal{E}}}\) choices. Recording which components of \(G_{\Pi}\) come from small blocks gives at most \(2^{m}\) further choices. Since \(\abs{\mathcal{E}}\le3m\), the total number of records is at most \(c_{\mathrm{rec}}^{m}\) after increasing the absolute constant.

It remains to check that one record determines at most one partition. The components marked small already give all blocks of size at most two. If \(k'=0\), we are done. If \(k'=1\), every point not in a small block belongs to the unique large block. If \(k'=2\), the line containing the segment in \(\mathcal{E}\), together with the two recorded sides, assigns every point of \(P\setminus S\) to one of the two large blocks.

Suppose now that \(k'\ge3\). The endpoints of a segment in \(\mathcal{E}\) determine its supporting line, and \(G_{\Pi}\) identifies the two sets \(B_{i}\cap S\) containing those endpoints. For \(i\in I\), the points of \(B_i\cap S\) are exactly the endpoints of \(\mathcal{E}\) on \(K_i\). Their coordinates determine their cyclic order on \(\partial K_i\). If several segments share an endpoint, their directions determine their order around that point, and the recorded sides determine on which side of each segment the polygon lies. Hence the record determines which segments occur, and in what order, on the boundary of each pseudotriangle, as well as which segments are interior.

The record therefore determines the screen graph, with degree-one vertices at the same point kept distinct as above. Orient each segment of \(\mathcal{E}\) from the endpoint with the smaller label to the other endpoint. Order the degree-one vertices first by the smaller endpoint label of their segments and then by the larger endpoint label. If two vertices come from the same segment, put first the one arising from the pseudotriangle on the left of the oriented segment. We perform the second step in this order. Hence the record determines the subdivision.

The subdivision has convex cells, and the polygons \(K_{i}\), \(i\in I\), lie in distinct cells. The recorded sides identify the cell containing each \(K_i\) and associate it with \(B_i\cap S\). Every subdivision edge lies on the supporting line of a segment of \(\mathcal{E}\), hence on a line through two points of \(S\). Since \(P\) is regular, no point of \(P\setminus S\) lies on the subdivision. Every point of \(P\setminus S\) lies in the interior of the cell that contains its block polygon, so we can assign it to the corresponding trace \(B_{i}\cap S\). Thus the record determines all large blocks and hence the whole partition.
\end{poc}

Put \(u=6k-6\). We apply Claims~\ref{claim:small-support} and~\ref{claim:discrete-record} to obtain
\[
N_{k}(P)\le\sum_{m=k}^{\min\{u,n\}}\binom{n}{m}c_{\mathrm{rec}}^{m}.
\]
If \(n\ge12k\), then \(u\le\frac{n}{2}\). For \(k\le m<u\), the ratio of two consecutive summands is \(c_{\mathrm{rec}}\frac{n-m}{m+1}>1\), so the sum is at most \(u\binom{n}{u}c_{\mathrm{rec}}^{u}\). We use \(k\le u\le6k\) and \(\binom{n}{u}\le(\frac{en}{u})^{u}\), and then enlarge \(c_{\mathrm{pc}}\) to obtain \(N_{k}(P)\le(\frac{c_{\mathrm{pc}}n}{k})^{6k}\). If \(n<12k\), then it is easy to see the sum is at most \((1+c_{\mathrm{rec}})^{n}\le(1+c_{\mathrm{rec}})^{12k}\), which gives the same bound after another enlargement of \(c_{\mathrm{pc}}\).

For \(k=1\) there is nothing to prove. Let \(k=2\). We apply Lemma~\ref{lem:strict-separation} to the two block hulls and obtain a separating line. We use the standard point-line duality, which sends each point of \(P\) to a line and each possible separating line to a point in the dual plane. As the dual point moves without crossing any of the \(n\) dual lines, no point of \(P\) changes side. Thus one cell of the subdivision formed by the dual lines corresponds to at most one bipartition of \(P\). Adding the dual lines one at a time shows that there are at most \(1+n+\binom{n}{2}\) cells. This is bounded by \((\frac{c_{\mathrm{pc}}n}{k})^{6k}\) after increasing \(c_{\mathrm{pc}}\). This finishes the proof.
\end{proof}

\subsection{Choosing the coloring}

We now use the counting results to select a coloring with no small crossing-free monochromatic partition.

\begin{lemma}\label{lem:planar-coloring}
There is an absolute constant \(c_{\mathrm{col}}>0\) such that every planar point set with at least \(c_{\mathrm{col}}t\) points has a red--blue coloring for which no crossing-free partition into monochromatic blocks has at most \(t\) red blocks and at most \(t\) blue blocks.
\end{lemma}

\begin{proof}[Proof of Lemma~\ref{lem:planar-coloring}]
We count the colorings that admit such a partition. For each \(k\), we first choose a crossing-free \(k\)-partition and then choose a color for each block. There are at most \(2^{k}\) choices in the second step, and every coloring under consideration occurs for some \(k\le2t\).

Take \(c_{\mathrm{col}}\) to be a sufficiently large positive integer, and first suppose that \(\abs{P}=c_{\mathrm{col}}t\). For \(1\le k<2t\), the ratio of the \((k+1)\)-st summand to the \(k\)-th summand is
\(2(\frac{c_{\mathrm{pc}}\abs{P}}{k+1})^{6}
(\frac{k}{k+1})^{6k}\).
Since \((1+\frac{1}{k})^{k}<e\) and \(k+1\le2t\), this ratio is at least
\(2(\frac{c_{\mathrm{pc}}c_{\mathrm{col}}}{2e})^{6}\), and hence is larger than \(1\) when \(c_{\mathrm{col}}\) is sufficiently large. Lemma~\ref{lem:planar-count} therefore gives
\[
\sum_{k=1}^{2t}\left[2\left(\frac{c_{\mathrm{pc}}\abs{P}}{k}\right)^{6}\right]^{k}
\le
2t\left[2\left(\frac{c_{\mathrm{pc}}c_{\mathrm{col}}}{2}\right)^{6}\right]^{2t}
<
2^{c_{\mathrm{col}}t}
=
2^{\abs{P}}.
\]
The last inequality holds for all \(t\ge1\) once
\(c_{\mathrm{col}}\) is chosen sufficiently large. Hence there exists some red--blue coloring of \(P\) admitting no forbidden partition. Both colors occur, since a monochromatic coloring admits a crossing-free partition with one block.

If \(\abs{P}>c_{\mathrm{col}}t\), choose a subset of \(c_{\mathrm{col}}t\) points, color it as above, and color the remaining points arbitrarily. Any forbidden partition of \(P\) would restrict, after deleting empty blocks, to a forbidden partition of the chosen subset. This proves the lemma.
\end{proof}

\subsection{Refining a bad coloring}

The last step converts disjoint \(s\)-convex covers into a crossing-free monochromatic partition with \(O(s^{2})\) blocks.

\begin{lemma}\label{lem:planar-refinement}
There is an absolute constant \(c_{\mathrm{ref}}>0\) such that every bad coloring for the parameter \(s\) has a crossing-free partition into at most \(c_{\mathrm{ref}}s^{2}\) monochromatic blocks.
\end{lemma}

\begin{proof}[Proof of Lemma~\ref{lem:planar-refinement}]
We apply Lemma~\ref{lem:component-reduction} to the red and blue covers. After discarding repetitions, we obtain nonempty compact convex sets \(K_{1},\ldots,K_{p}\) and \(L_{1},\ldots,L_{q}\), where \(p,q\le s\), such that the red points lie in \(\bigcup_{i=1}^{p}K_{i}\), the blue points lie in \(\bigcup_{j=1}^{q}L_{j}\), and \(K_{i}\cap L_{j}=\emptyset\) for every \(i,j\).

For each pair \((i,j)\), we apply Lemma~\ref{lem:strict-separation}. After rescaling the resulting affine function \(\ell_{i,j}:\R^{2}\to\R\), we may assume that \(\ell_{i,j}\ge2\) on \(K_{i}\) and \(\ell_{i,j}\le-2\) on \(L_{j}\). Choose a closed triangle \(T\) that contains all these sets in its interior. Define \(R_{i}=T\cap\bigcap_{j=1}^{q}\{\ell_{i,j}\ge1\}\) and \(B_{j}=T\cap\bigcap_{i=1}^{p}\{\ell_{i,j}\le-1\}\). These are convex polygons with \(K_{i}\subseteq\operatorname{int}(R_{i})\), \(L_{j}\subseteq\operatorname{int}(B_{j})\), and \(R_{i}\cap B_{j}=\emptyset\). Let \(m\) be their total number of vertices, counting a common vertex once for each polygon in which it occurs. Since \(R_{i}\) is cut out inside \(T\) by \(q\) halfplanes and \(B_{j}\) by \(p\) halfplanes,
\[
m\le p(q+3)+q(p+3)=2pq+3p+3q.
\]

Put \(R=\bigcup_{i=1}^{p}R_{i}\) and \(B=\bigcup_{j=1}^{q}B_{j}\). These compact polygonal sets are disjoint. We apply Lemma~\ref{lem:shortest-fence} and choose a shortest fence separating them. Call the two sides of the fence red and blue according to whether they contain \(R\) or \(B\).

We now view the fence as a plane graph. First split its segments at every intersection. Delete any edge that has the same color on both sides, and merge two consecutive collinear edges whenever their common vertex has degree two. The remaining edges form a plane straight-line graph \(G\). A \emph{face} of \(G\) is a connected component of \(\R^{2}\setminus G\), every face is red or blue. Let \(e_{G},v_{G},f_{G}\), and \(\kappa_{G}\) be the numbers of edges, vertices, faces, and connected components of \(G\), respectively.

\begin{claim}\label{claim:planar-fence}
 \(e_{G}\le m+2(p+q)\).
\end{claim}

\begin{poc}
Every edge separates a red face from a blue face. In particular, no edge is a \emph{bridge}, meaning an edge whose removal increases the number of connected components.

We first bound the number of faces. Suppose that a face contains no interior point of any \(R_{i}\) or \(B_{j}\). Every adjacent face has the opposite color, so we may recolor the empty face and remove its boundary. This produces a shorter fence that still separates \(R\) from \(B\), a contradiction. Thus every face contains the interior of at least one polygon. Conversely, the interior of each polygon is connected and does not meet the fence, so it lies in a single face. Therefore \(f_{G}\le p+q\).

Let \(V_{0}\) be the set of vertices of the polygons \(R_{i}\) and \(B_{j}\), so \(\abs{V_{0}}\le m\). We claim that every degree-two vertex of \(G\) belongs to \(V_{0}\). Suppose otherwise, and let \(v\notin V_{0}\) have degree two. Its two edges are not collinear, since we already merged collinear edges. Choose a small disk \(D\) around \(v\) that contains no other vertex of \(G\) and meets \(G\) only in the two edges at \(v\). Let \(x\) and \(y\) be the two points where these edges leave \(D\). If \(v\) lies outside all polygons, we may choose \(D\) disjoint from them. If \(v\) lies on the boundary of the red union, we choose \(D\) disjoint from the blue union and small enough that the part of \(D\) outside the red union is an intersection of halfplanes and hence convex. The blue case is symmetric.

Replace the broken path \([x,v]\cup[v,y]\) by the straight segment \([x,y]\), and give the region between the old and new paths the color of the face containing \([x,y]\setminus\{x,y\}\). This preserves the separation of all polygons and strictly decreases the length, contradicting the choice of the fence. Thus every degree-two vertex lies in \(V_{0}\), and there are at most \(m\) such vertices.

Every other vertex has even degree at least four, because the face colors alternate around it. Let \(h\) be the number of these vertices. Euler's formula and the identity \(\sum_{v\in V(G)}\deg_{G}(v)=2e_{G}\) give
\[
e_{G}-v_{G}=f_{G}-1-\kappa_{G}=\frac{1}{2}\sum_{v\in V(G)}(\deg_{G}(v)-2)\ge h.
\]
Since \(v_{G}\le m+h\), we obtain \(e_{G}\le m+2(f_{G}-1-\kappa_{G})\le m+2(p+q)\).
\end{poc}

We now have \(e_{G}\le2pq+5p+5q\le12s^{2}\). Every point of \(P\) lies in the interior of a polygon of its color, so it lies outside \(G\) and belongs to a face of that color. Choose a large closed triangle \(\Delta\) containing \(G\) and \(P\). We say that a triangulation of \(\Delta\) \emph{respects \(G\)} if every edge of \(G\) is a union of edges of the triangulation. By adding edges and, when needed, vertices on existing edges, we obtain such a triangulation. Since every vertex of \(G\) has degree at least two, \(v_{G}\le e_{G}\). We may therefore complete \(G\) and the boundary of \(\Delta\) to a triangulation with \(O(e_{G}+1)\) closed triangles. Each triangle lies in the closure of one face of \(G\).

Fix an order of the triangles and assign each point of \(P\) to the first triangle containing it. The points assigned to one triangle have the color of the face containing its interior, so every resulting block is monochromatic.

We check that the convex hulls of distinct nonempty blocks are disjoint. Let \(P_{a}\) and \(P_{b}\) be assigned to triangles \(\Delta_{a}\) and \(\Delta_{b}\), with \(\Delta_{a}\) earlier in the order. The two triangles are disjoint or meet in a common vertex or edge \(\sigma\). In the latter case, the assignment rule implies \(P_{b}\cap\sigma=\emptyset\). Choose an affine function whose maximum on \(\Delta_{b}\) is attained exactly on \(\sigma\). Its value on every point of the finite set \(P_{b}\) is strictly smaller than this maximum, and the same is true throughout \(\conv(P_{b})\). Hence \(\conv(P_{b})\cap\sigma=\emptyset\). Since \(\conv(P_{a})\subseteq\Delta_{a}\) and \(\conv(P_{b})\subseteq\Delta_{b}\), their convex hulls are disjoint.

After discarding empty blocks, we obtain a crossing-free monochromatic partition with \(O(e_{G}+1)=O(s^{2})\) blocks. By increasing the absolute constant, we may bound their number by \(c_{\mathrm{ref}}s^{2}\).
\end{proof}

\subsection{Completing the proof}

\begin{proof}[Proof of Theorem~\ref{thm:planar}]
It remains to prove the upper bound. Put \(t=\lceil c_{\mathrm{ref}}s^{2}\rceil\), where \(c_{\mathrm{ref}}\) is the constant in Lemma~\ref{lem:planar-refinement}. We apply Lemma~\ref{lem:planar-coloring} to a planar point set of size at least \(c_{\mathrm{col}}t\) and obtain a red--blue coloring with no crossing-free monochromatic partition having at most \(t\) blocks of each color. If the coloring were bad, we could apply Lemma~\ref{lem:planar-refinement} to obtain such a partition with at most \(c_{\mathrm{ref}}s^{2}\le t\) blocks in total, a contradiction. Hence the coloring is not bad and \(f_{2}(2,s,s)=O(s^{2})\). This proves the theorem.
\end{proof}

\section{Lower bounds}

We now prove Theorems \ref{thm:lower} and \ref{thm:power-lower}. Two geometric operations will be used. The first places affine copies of a point set near distinct vertices of a polytope and produces a two-part construction without increasing the dimension. Every vertex of a polytope is exposed, meaning that some linear functional has it as its unique maximizer. The second operation places affine copies of the current point set inside different facets of a higher-dimensional polytope and adds one further part at the cost of one dimension.

We first give a precise form of the property that will be preserved. Let \(X\subseteq\R^{d}\) be finite, let \(k\ge2\), and let \(s_{1},\ldots,s_{k}\ge1\). We say that \(X\) has property \(\mathsf{B}(s_{1},\ldots,s_{k})\) if the following holds. For every map \(\chi:X\to[k]\), write \(X_{i}=\chi^{-1}(i)\). There are partitions
\[
X_{i}=X_{i,1}\sqcup\cdots\sqcup X_{i,s_{i}}
\]
for \(i\in[k]\), where some sets may be empty, such that for every \(\boldsymbol{a}=(a_{1},\ldots,a_{k})\in\prod_{i=1}^{k}[s_{i}]\),
\[
\bigcap_{i=1}^{k}\operatorname{conv}(X_{i,a_{i}})=\emptyset.
\]
The order of the quantifiers is important: after \(\chi\) is fixed, we choose the partitions, and those fixed partitions must satisfy the displayed condition simultaneously for every \(\boldsymbol{a}\).
The point of requiring this for every choice of one component from each part is that it immediately gives a bad point set for \(f_{k}\). Indeed, put \(C_{i}=\bigcup_{a=1}^{s_{i}}\operatorname{conv}(X_{i,a})\). Then \(C_{i}\) is \(s_{i}\)-convex and contains \(X_{i}\), while
\[
\bigcap_{i=1}^{k}C_{i}=\bigcup_{\boldsymbol{a}\in\prod_{i=1}^{k}[s_{i}]}\bigcap_{i=1}^{k}\operatorname{conv}(X_{i,a_{i}})=\emptyset.
\]
Consequently, if \(X\) has property \(\mathsf{B}(s_{1},\ldots,s_{k})\), then \(f_{k}(d,s_{1},\ldots,s_{k})>\abs{X}\).

The next lemma is the new step that saves one dimension. Unlike the facet lifting in Lemma~\ref{lem:facet-lifting}, which increases the ambient dimension when a new part is added, it creates the first two parts in the same space \(\R^{\ell}\). Its assumption restricts the directions of the halfspaces, and Lemma~\ref{lem:polytope-vc} provides exactly this additional property.

\begin{lemma}\label{lem:two-part-amplification}
Let \(\ell\ge2\), and let \(X\subseteq\R^{\ell}\) be finite. Suppose that every subset of \(X\) is the intersection of \(X\) with a union of at most \(s\) closed halfspaces, each having a normal vector whose coordinates are strictly positive. For every integer \(t\ge1\), there is a set \(X'\subseteq\R^{\ell}\) of \(t\abs{X}\) points having property \(\mathsf{B}(s,t)\).
\end{lemma}

\begin{proof}[Proof of Lemma~\ref{lem:two-part-amplification}]
The geometric idea is to anchor \(t\) tiny affine copies of \(X\) at distinct vertices of a polytope. For a halfspace that realizes a prescribed trace on one copy, we transform its normal vector into the interior of the normal cone at that vertex. The transformed halfspace cuts the designated copy exactly as prescribed and contains every other copy. We can therefore group the first color according to the original \(s\) halfspaces and group the second color according to the \(t\) host vertices.

Write \(h(\boldsymbol{u},c)=\{\boldsymbol{x}:\langle\boldsymbol{u},\boldsymbol{x}\rangle\le c\}\). For every subset \(Y\subseteq X\), choose one representation of \(Y\) by at most \(s\) halfspaces as in the assumption. If fewer than \(s\) are used, add copies of a fixed halfspace with coordinatewise positive normal vector that contains no point of \(X\). We may therefore assume that exactly \(s\) halfspaces have been chosen for every \(Y\). Since \(X\) is finite, only finitely many pairs \((\boldsymbol{u},c)\) occur in all these chosen representations.

To implement this idea, choose a full-dimensional convex polytope \(Q\subseteq\R^{\ell}\) with distinct vertices \(\boldsymbol{v}_{1},\ldots,\boldsymbol{v}_{t}\). For each \(j\in[t]\), define
\[
N_{Q}(\boldsymbol{v}_{j})=\{\boldsymbol{w}\in\R^{\ell}:\langle\boldsymbol{w},\boldsymbol{z}-\boldsymbol{v}_{j}\rangle\le0\text{ for every }\boldsymbol{z}\in Q\}.
\]
This is the set of vectors whose corresponding linear functions are maximized over \(Q\) at \(\boldsymbol{v}_{j}\). Since \(\boldsymbol{v}_{j}\) is a vertex, the outward normal vectors of the facets containing it span \(\R^{\ell}\), their positive combinations show that \(N_{Q}(\boldsymbol{v}_{j})\) has nonempty interior. Choose linearly independent vectors \(\boldsymbol{w}_{j,1},\ldots,\boldsymbol{w}_{j,\ell}\) in this interior, and let \(L_{j}\) be the invertible linear map satisfying \(L_{j}\boldsymbol{e}_{b}=\boldsymbol{w}_{j,b}\) for every \(b\in[\ell]\). If every coordinate of \(\boldsymbol{u}\) is positive, then
\[
L_{j}\boldsymbol{u}=\sum_{b=1}^{\ell}u_{b}\boldsymbol{w}_{j,b}\in\operatorname{int}N_{Q}(\boldsymbol{v}_{j}).
\]
Here, \(\operatorname{int}N_{Q}(\boldsymbol{v}_{j})\) denotes the
interior of the normal cone in \(\R^{\ell}\). In particular, \(\boldsymbol{v}_{j}\) is the unique vertex of \(Q\) at which the linear function with normal vector \(L_{j}\boldsymbol{u}\) is maximized.

For \(\varepsilon>0\), put
\[
X^{(j)}=\{\boldsymbol{v}_{j}+\varepsilon L_{j}^{-\top}\boldsymbol{x}:\boldsymbol{x}\in X\}.
\]
Here \(L_{j}^{-\top}\) denotes the inverse transpose of \(L_{j}\).
For every chosen halfspace \(h(\boldsymbol{u},c)\), define
\[
\widehat H_{j}(\boldsymbol{u},c)=\{\boldsymbol{y}\in\R^{\ell}:\langle L_{j}\boldsymbol{u},\boldsymbol{y}-\boldsymbol{v}_{j}\rangle\le\varepsilon c\}.
\]
If \(\boldsymbol{y}=\boldsymbol{v}_{j}+\varepsilon L_{j}^{-\top}\boldsymbol{x}\), then
\[
\langle L_{j}\boldsymbol{u},\boldsymbol{y}-\boldsymbol{v}_{j}\rangle=\varepsilon\langle\boldsymbol{u},\boldsymbol{x}\rangle.
\]
Thus \(\widehat H_{j}(\boldsymbol{u},c)\) has exactly the same intersection with \(X^{(j)}\) as \(h(\boldsymbol{u},c)\) has with \(X\), under the affine identification between these two sets.

The same halfspace also contains every other copy when \(\varepsilon\) is sufficiently small. Indeed, if \(b\ne j\), then \(L_{j}\boldsymbol{u}\in\operatorname{int}N_{Q}(\boldsymbol{v}_{j})\) gives
\(
\langle L_{j}\boldsymbol{u},\boldsymbol{v}_{b}-\boldsymbol{v}_{j}\rangle<0.
\)
Consequently, for every \(\boldsymbol{x}\in X\), one has \(\boldsymbol{v}_{b}+\varepsilon L_{b}^{-\top}\boldsymbol{x}\in\widehat H_{j}(\boldsymbol{u},c)\) once \(\varepsilon>0\) is small enough. There are only finitely many vertices, points, and chosen halfspaces, so one value of \(\varepsilon\) works simultaneously for all of them. We also take it small enough that the sets \(X^{(1)},\ldots,X^{(t)}\) are pairwise disjoint, and put \(X'=X^{(1)}\sqcup\cdots\sqcup X^{(t)}\).

Fix a map \(\chi:X'\to[2]\), and put \(A=\chi^{-1}(1)\) and \(B=\chi^{-1}(2)\). For each \(j\in[t]\), pull the set \(A\cap X^{(j)}\) back to a subset of \(X\), and use its chosen representation by \(s\) halfspaces. Let \(\widehat H_{j,1},\ldots,\widehat H_{j,s}\) be the corresponding halfspaces in \(\R^{\ell}\). Their union contains exactly the points of \(A\cap X^{(j)}\) inside the \(j\)-th copy. Assign each such point to one halfspace that contains it, and let \(A_{j,a}\) be the set assigned to \(\widehat H_{j,a}\). Define \(A_{a}=\bigcup_{j=1}^{t}A_{j,a}\) for \(a\in[s]\), and define \(B_{j}=B\cap X^{(j)}\) for \(j\in[t]\). These sets partition \(A\) and \(B\), respectively.

The following claim verifies the required component-by-component disjointness.

\begin{claim}\label{claim:two-part-disjoint}
For every \(a\in[s]\) and \(j\in[t]\), one has \(\operatorname{conv}(A_{a})\cap\operatorname{conv}(B_{j})=\emptyset\).
\end{claim}

\begin{poc}
By construction, \(A_{j,a}\subseteq\widehat H_{j,a}\). Moreover, every copy \(X^{(b)}\) with \(b\ne j\) is contained in \(\widehat H_{j,a}\). Hence \(A_{a}\subseteq\widehat H_{j,a}\), and the convexity of this halfspace gives \(\operatorname{conv}(A_{a})\subseteq\widehat H_{j,a}\).

Inside \(X^{(j)}\), the union of \(\widehat H_{j,1},\ldots,\widehat H_{j,s}\) contains exactly the points of \(A\cap X^{(j)}\). Therefore every point of \(B_{j}\) lies in the open halfspace \(\R^{\ell}\setminus\widehat H_{j,a}\). This open halfspace is convex, so it also contains \(\operatorname{conv}(B_{j})\). The two convex hulls are therefore disjoint.
\end{poc}

Claim \ref{claim:two-part-disjoint} says exactly that the partitions \(A=A_{1}\sqcup\cdots\sqcup A_{s}\) and \(B=B_{1}\sqcup\cdots\sqcup B_{t}\) satisfy property \(\mathsf{B}(s,t)\).
\end{proof}

We now use a separate lifting operation. It places affine copies of a point set having property \(\mathsf{B}\) inside facets of a polytope. The exposed-face property ensures that a convex combination lying on one chosen facet can use points only from the copy on that facet, which is exactly what allows one more part to be introduced.

\begin{lemma}\label{lem:facet-lifting}
Suppose that a finite set \(X\subseteq\R^{d}\) has property \(\mathsf{B}(s_{1},\ldots,s_{k})\). For every integer \(t\ge1\), there is a set \(X'\subseteq\R^{d+1}\) of \(t\abs{X}\) points having property \(\mathsf{B}(s_{1},\ldots,s_{k},t)\).
\end{lemma}

\begin{proof}[Proof of Lemma~\ref{lem:facet-lifting}]
Choose a full-dimensional convex polytope \(Q\subseteq\R^{d+1}\) with distinct facets \(F_{1},\ldots,F_{t}\), and place an affine copy \(X^{(j)}\) of \(X\) inside each \(F_{j}\), away from its boundary. Set \(X'=X^{(1)}\sqcup\cdots\sqcup X^{(t)}\).

Fix a map \(\chi:X'\to[k+1]\). On each copy \(X^{(j)}\), temporarily move every point in part \(k+1\) to part \(1\). This gives a map from \(X^{(j)}\) to \([k]\), so property \(\mathsf{B}(s_{1},\ldots,s_{k})\) gives partitions into sets \(Y^{(j)}_{i,a}\) with empty intersection for every choice of one component from each of the \(k\) parts.

We now remove the temporarily moved points from the first part. Put \(Z^{(j)}_{1,a}=Y^{(j)}_{1,a}\cap\chi^{-1}(1)\), and put \(Z^{(j)}_{i,a}=Y^{(j)}_{i,a}\) for \(2\le i\le k\). These sets partition the points of the original parts \(1,\ldots,k\) inside \(X^{(j)}\). Since each convex hull has only been made smaller, for every \((a_{1},\ldots,a_{k})\in\prod_{i=1}^{k}[s_{i}]\),
\[
\bigcap_{i=1}^{k}\operatorname{conv}(Z^{(j)}_{i,a_{i}})\subseteq\bigcap_{i=1}^{k}\operatorname{conv}(Y^{(j)}_{i,a_{i}})=\emptyset.
\]

For the first \(k\) parts, join the corresponding sets from all facets by defining \(Z_{i,a}=\bigcup_{j=1}^{t}Z^{(j)}_{i,a}\). For the new part, use one set on each facet: define \(Z_{k+1,j}=\chi^{-1}(k+1)\cap X^{(j)}\). These sets give the required partitions of all \(k+1\) parts.

It remains to check the intersections of their convex hulls. Fix \(a_{i}\in[s_{i}]\) for \(i\in[k]\) and \(j\in[t]\), and suppose that
\[
\boldsymbol{z}\in\bigcap_{i=1}^{k}\operatorname{conv}(Z_{i,a_{i}})\cap\operatorname{conv}(Z_{k+1,j}).
\]
The last convex hull is contained in \(F_{j}\), so \(\boldsymbol{z}\in F_{j}\). There is an affine function \(\lambda_{j}:\R^{d+1}\to\R\) such that \(\lambda_{j}\le1\) on \(Q\) and equality holds precisely on \(F_{j}\). Thus every point of \(X^{(j)}\) has \(\lambda_{j}\)-value \(1\), while every point in the other copies has value strictly less than \(1\). For each \(i\in[k]\), expressing \(\boldsymbol{z}\) as a convex combination of points of \(Z_{i,a_{i}}\) and applying \(\lambda_{j}\) shows that only points from \(X^{(j)}\) can have positive coefficients. Hence \(\boldsymbol{z}\in\operatorname{conv}(Z^{(j)}_{i,a_{i}})\) for every \(i\in[k]\), contradicting the empty intersection proved above. Therefore \(X'\) has property \(\mathsf{B}(s_{1},\ldots,s_{k},t)\).
\end{proof}

We can now prove the sharp lower bound.

\begin{proof}[Proof of Theorem~\ref{thm:lower}]
Set \(\ell=d-r+2\), so \(\ell\ge4\). We apply Lemma~\ref{lem:polytope-vc} and obtain a set \(X\subseteq\R^{\ell}\) with
\[
\abs{X}\ge c_{0}\ell s\log(s+1)
\]
such that every subset is realized by a union of at most \(s\) halfspaces with coordinatewise positive normal vectors. We apply Lemma~\ref{lem:two-part-amplification} with \(t=s\) and obtain a set in the same space \(\R^{\ell}\) having property \(\mathsf{B}(s,s)\) and containing \(s\abs{X}\) points. We then apply Lemma~\ref{lem:facet-lifting} \(r-2\) times, taking \(t=s\) each time. The resulting set \(X_{r}\) lies in
\(
\R^{\ell+r-2}=\R^{d},
\)
has property \(\mathsf{B}(s,\ldots,s)\), and satisfies
\[
\abs{X_{r}}=s^{r-1}\abs{X}\ge c_{0}(d-r+2)s^{r}\log(s+1).
\]
As observed at the beginning of the section, every partition of \(X_{r}\) into \(r\) nonempty parts is bad. Therefore \(f_{r}(d,s,\ldots,s)>\abs{X_{r}}\), and the theorem follows with \(c=c_{0}\).
\end{proof}

Finally, we apply the lifting lemma to the planar construction in Lemma~\ref{lem:planar-bad-set}.

\begin{proof}[Proof of Theorem~\ref{thm:power-lower}]
Put \(k=\min\{r,d\}\). We apply Lemma~\ref{lem:planar-bad-set} and obtain a set of \(s^{2}\) points in \(\R^{2}\) having property \(\mathsf{B}(s,s)\). Applying Lemma~\ref{lem:facet-lifting} \(k-2\) times, with \(t=s\) at every step, gives a set \(X\subseteq\R^{k}\) of \(s^{k}\) points having property \(\mathsf{B}(s,\ldots,s)\) for \(k\) parts. We may regard \(X\) as a subset of \(\R^{d}\).

If \(r\le d\), then \(k=r\), and the definition of property \(\mathsf{B}\) gives \(f_{r}(d,s,\ldots,s)>s^{r}\). It remains to consider \(r>d\), so \(k=d\). Consider any surjective map \(\chi:X\to[r]\), if no such map exists, then \(X\) is already a bad example. Choose a surjection \(\phi:[r]\to[d]\), and merge the \(r\) parts into \(d\) parts by setting \(Y_{j}=\bigcup_{i\in\phi^{-1}(j)}\chi^{-1}(i)\). Property \(\mathsf{B}(s,\ldots,s)\) gives \(s\)-convex sets \(C_{1},\ldots,C_{d}\) with \(Y_{j}\subseteq C_{j}\) and \(\bigcap_{j=1}^{d}C_{j}=\emptyset\). For each \(i\in[r]\), set \(D_{i}=C_{\phi(i)}\). Then \(D_{i}\) is \(s\)-convex and contains \(\chi^{-1}(i)\). Since \(\phi\) is surjective,
\[
\bigcap_{i=1}^{r}D_{i}=\bigcap_{j=1}^{d}C_{j}=\emptyset.
\]
Thus every \(r\)-partition of \(X\) is bad, and \(f_{r}(d,s,\ldots,s)>\abs{X}=s^{d}\).
\end{proof}

\section{General upper bounds}
\subsection{Proof of Theorem~\ref{thm:general-upper}}
We now give the proof of Theorem~\ref{thm:general-upper}. To make the two counting constructions self-contained, we recall the notation used throughout this section. Put \(\boldsymbol{s}=(s_{1},\ldots,s_{r})\), \(S(\boldsymbol{s})=\prod_{i=1}^{r}s_{i}\), and \(q=\min\{r,d+1\}\). We also write
\[
H_{d,r}(\boldsymbol{s})=\sum_{\substack{\emptyset\ne I\subseteq[r]\\\abs{I}\le d+1}}\abs{I}\prod_{i\in I}s_{i}
\]
and \(T_{d,r}(\boldsymbol{s})=\min\{qS(\boldsymbol{s}),H_{d,r}(\boldsymbol{s})\}\). When a polyhedron is presented as an intersection of closed halfspaces, each listed proper halfspace is one \emph{halfspace occurrence}. Repetitions in different polyhedra count separately, while a factor equal to \(\R^{d}\) contributes no occurrence.

For completeness, we first derive the familiar trace bound for halfspaces from Radon's theorem and the Sauer--Shelah lemma.

\begin{lemma}\label{lem:halfspace-subsets}
Let \(P\subseteq\R^{d}\) have \(m\) points, and let \(\tau_{d}(P)\) be the number of subsets of the form \(P\cap H\), where \(H\) is a closed halfspace. Then \(\tau_{d}(P)\le\sum_{j=0}^{d+1}\binom{m}{j}\). If \(m\ge d+1\), then \(\tau_{d}(P)\le \left(\frac{em}{d+1}\right)^{d+1}\).
\end{lemma}
\begin{proof}[Proof of Lemma~\ref{lem:halfspace-subsets}]
We first apply Radon's theorem in Lemma~\ref{lem:radon} to see that closed halfspaces in \(\R^{d}\) shatter no set of size \(d+2\). Indeed, split any \(d+2\) points into two nonempty parts \(A\) and \(B\) with \(\operatorname{conv}(A)\cap\operatorname{conv}(B)\ne\emptyset\). If a closed halfspace contained exactly the points of \(A\), its open complement would contain the points of \(B\). Convexity would then put \(\operatorname{conv}(A)\) in the halfspace and \(\operatorname{conv}(B)\) in its complement, a contradiction. We now apply the Sauer--Shelah lemma in Lemma~\ref{lem:sauer} to obtain the first bound. The second follows from \(\sum_{j=0}^{v}\binom{m}{j}\le \left(\frac{em}{v}\right)^{v}\) with \(v=d+1\).
\end{proof}

A map \(\chi:P\to [r]\) is called \emph{bad} if, with \(P_{i}=\chi^{-1}(i)\), there are \(s_{i}\)-convex sets \(C_{i}\supseteq P_{i}\) such that \(\bigcap_{i=1}^{r}C_{i}=\emptyset\). We use Lemma~\ref{lem:component-reduction} to replace the components of these sets by compact convex hulls of assigned points.

We next turn those compact convex sets into polyhedra. Here is the underlying mechanism. For each full component choice, Helly's theorem retains at most \(q\) part indices that still witness emptiness, and simultaneous separation replaces the corresponding bodies by halfspaces with empty intersection. We then intersect all halfspaces attached to the same component. There are two ways to organize this information: the \emph{full-tuple alternative} treats all \(S(\boldsymbol{s})\) full choices separately and costs at most \(qS(\boldsymbol{s})\) occurrences, the \emph{sparse-subtuple alternative} pre-enumerates every choice on at most \(d+1\) part indices and costs at most \(H_{d,r}(\boldsymbol{s})\) occurrences.

\begin{lemma}\label{lem:polyhedral-reduction}
Let \(P_{1},\ldots,P_{r}\) be nonempty finite subsets of \(\R^{d}\). Suppose that, for every \(i\in [r]\) and \(a\in [s_{i}]\), there is a nonempty compact convex set \(D_{i,a}\) such that \(P_{i}\subseteq\bigcup_{a=1}^{s_{i}}D_{i,a}\), and for every tuple \(\boldsymbol{a}=(a_{1},\ldots,a_{r})\in\prod_{i=1}^{r}[s_{i}]\) one has \(\bigcap_{i=1}^{r}D_{i,a_{i}}=\emptyset\). Then one can choose sets \(Q_{i}\supseteq P_{i}\) such that \(\bigcap_{i=1}^{r}Q_{i}=\emptyset\), each \(Q_{i}\) is a union of at most \(s_{i}\) convex polyhedra, and at most \(T_{d,r}(\boldsymbol{s})\) halfspace occurrences are used. More precisely, the full-tuple alternative uses at most \(qS(\boldsymbol{s})\) occurrences and, for every such \(\boldsymbol{a}\), specifies a set \(I(\boldsymbol{a})\subseteq[r]\) with \(1\le\abs{I(\boldsymbol{a})}\le q\) and halfspaces \(H_{i,\boldsymbol{a}}\supseteq D_{i,a_{i}}\), \(i\in I(\boldsymbol{a})\), such that both \(\bigcap_{i\in I(\boldsymbol{a})}D_{i,a_{i}}\) and \(\bigcap_{i\in I(\boldsymbol{a})}H_{i,\boldsymbol{a}}\) are empty. The sparse-subtuple alternative uses at most \(H_{d,r}(\boldsymbol{s})\) occurrences.
\end{lemma}
\begin{proof}[Proof of Lemma~\ref{lem:polyhedral-reduction}]
There are two useful ways to turn the sets \(D_{i,a}\) into halfspaces. We begin with the full-tuple alternative described above.

    \begin{claim}\label{claim:full-tuple-construction}
There are sets \(Q_{1},\ldots,Q_{r}\) satisfying the conclusion of the lemma with at most \(qS(\boldsymbol{s})\) halfspace occurrences. For every full component choice \(\boldsymbol{a}\), the construction also specifies a set \(I(\boldsymbol{a})\) and halfspaces \(H_{i,\boldsymbol{a}}\), \(i\in I(\boldsymbol{a})\), satisfying all the witness conditions in the statement of the lemma.
\end{claim}
\begin{poc}
Fix a full component choice \(\boldsymbol{a}\). By assumption, the sets \(D_{1,a_{1}},\ldots,D_{r,a_{r}}\) have empty intersection. We apply Helly's theorem in Lemma~\ref{lem:helly} and choose \(I(\boldsymbol{a})\subseteq [r]\) with \(1\le\abs{I(\boldsymbol{a})}\le q\) and \(\bigcap_{i\in I(\boldsymbol{a})}D_{i,a_{i}}=\emptyset\). We then apply the simultaneous separation statement in Lemma~\ref{lem:separation} and choose halfspaces \(H_{i,\boldsymbol{a}}\supseteq D_{i,a_{i}}\), \(i\in I(\boldsymbol{a})\), whose intersection is empty.

For fixed \(i\) and \(a\), define \(Q_{i,a}\) as the intersection of all halfspaces \(H_{i,\boldsymbol{b}}\) with \(b_{i}=a\) and \(i\in I(\boldsymbol{b})\). If there is no such halfspace, set \(Q_{i,a}=\R^{d}\). Finally put \(Q_{i}=\bigcup_{a=1}^{s_{i}}Q_{i,a}\). Each \(Q_{i}\) is a union of at most \(s_{i}\) convex polyhedra.

We need to check that \(P_{i}\subseteq Q_{i}\) first. Take \(\boldsymbol{p}\in P_{i}\). By assumption, \(\boldsymbol{p}\) lies in at least one of the sets \(D_{i,a}\), fix such an \(a\). Every halfspace used in the definition of \(Q_{i,a}\) contains \(D_{i,a}\), so it contains \(\boldsymbol{p}\). Thus \(\boldsymbol{p}\in Q_{i,a}\subseteq Q_{i}\). Secondly, we check that \(\bigcap_{i=1}^{r}Q_{i}=\emptyset\). If \(\boldsymbol{x}\) belonged to all \(Q_{i}\), then for each \(i\) we could choose \(a_{i}\) with \(\boldsymbol{x}\in Q_{i,a_{i}}\). Let \(\boldsymbol{a}=(a_{1},\ldots,a_{r})\). For every \(i\in I(\boldsymbol{a})\), the halfspace \(H_{i,\boldsymbol{a}}\) appears in the intersection defining \(Q_{i,a_{i}}\). Hence \(\boldsymbol{x}\) lies in all the halfspaces \(H_{i,\boldsymbol{a}}\), \(i\in I(\boldsymbol{a})\), contradicting how those halfspaces were chosen.

Thus the construction uses at most \(\sum_{\boldsymbol{a}}\abs{I(\boldsymbol{a})}\le qS(\boldsymbol{s})\) halfspace occurrences, and the chosen sets \(I(\boldsymbol{a})\) and halfspaces have all the required witness properties.
\end{poc}

The second construction lists only small sets \(I\subseteq [r]\) from the start. It is useful when \(r\) is large, because then we do not need to store information indexed by all \(r\) parts at once.

\begin{claim}\label{claim:sparse-construction}
There are sets \(Q_{1},\ldots,Q_{r}\) such that \(P_{i}\subseteq Q_{i}\) for every \(i\), \(\bigcap_{i=1}^{r}Q_{i}=\emptyset\), each \(Q_{i}\) is a union of at most \(s_{i}\) convex polyhedra, and the total number of halfspace occurrences is at most \(H_{d,r}(\boldsymbol{s})\).
\end{claim}
\begin{poc}
For every nonempty \(I\subseteq [r]\) with \(\abs{I}\le d+1\) and every component choice \(\boldsymbol{a}_{I}=(a_{i})_{i\in I}\in\prod_{i\in I}[s_{i}]\), define sets \(K_{i,I,\boldsymbol{a}_{I}}\), \(i\in I\), as follows. If \(\bigcap_{i\in I}D_{i,a_{i}}=\emptyset\), we apply Lemma~\ref{lem:separation} and obtain closed halfspaces \(K_{i,I,\boldsymbol{a}_{I}}\supseteq D_{i,a_{i}}\), \(i\in I\), with empty intersection. If \(\bigcap_{i\in I}D_{i,a_{i}}\ne\emptyset\), set \(K_{i,I,\boldsymbol{a}_{I}}=\R^{d}\) for every \(i\in I\). These sets impose no restriction and contribute no halfspace occurrence.

For fixed \(i\) and \(a\), define \(Q_{i,a}\) as the intersection of all sets \(K_{i,I,\boldsymbol{a}_{I}}\) with \(i\in I\) and \(a_{i}=a\), and set \(Q_{i}=\bigcup_{a=1}^{s_{i}}Q_{i,a}\). Each \(Q_{i}\) is a union of at most \(s_{i}\) convex polyhedra. Also \(P_{i}\subseteq Q_{i}\), because every factor in the intersection defining \(Q_{i,a}\) contains \(D_{i,a}\) when it is a halfspace, and equals \(\R^{d}\) otherwise.

It remains to see that the \(Q_{i}\) have empty total intersection. Suppose \(\boldsymbol{x}\in\bigcap_{i=1}^{r}Q_{i}\). Choose \(a_{i}\) with \(\boldsymbol{x}\in Q_{i,a_{i}}\) for each \(i\). By assumption, the full intersection \(\bigcap_{i=1}^{r}D_{i,a_{i}}\) is empty. We apply Helly's theorem in Lemma~\ref{lem:helly} and obtain a set \(I\subseteq [r]\) with \(\abs{I}\le d+1\) and \(\bigcap_{i\in I}D_{i,a_{i}}=\emptyset\). For \(\boldsymbol{a}_{I}=(a_{i})_{i\in I}\), the sets \(K_{i,I,\boldsymbol{a}_{I}}\), \(i\in I\), are closed halfspaces with empty intersection. Since \(\boldsymbol{x}\in Q_{i,a_{i}}\) for every \(i\in I\), the point \(\boldsymbol{x}\) lies in all these halfspaces, a contradiction.

Moreover, the number of halfspace occurrences is at most the number of possible triples \((I,\boldsymbol{a}_{I},i)\), namely \(H_{d,r}(\boldsymbol{s})\).
\end{poc}
We use Claim~\ref{claim:full-tuple-construction} for the full-tuple alternative and Claim~\ref{claim:sparse-construction} for the sparse-subtuple alternative. We choose the first when \(qS(\boldsymbol{s})\le H_{d,r}(\boldsymbol{s})\) and the second otherwise. The chosen construction uses at most \(T_{d,r}(\boldsymbol{s})\) halfspace occurrences, which proves the lemma.
\end{proof}
We now record only the information that can be seen on the fixed point set \(P\). A halfspace \(H\) is recorded by the subset \(P\cap H\), and we use \(\mathcal{D}\) to denote the record. In the sparse-subtuple alternative, it is the list of subsets \(P\cap K_{i,I,\boldsymbol{a}_{I}}\), one for every triple \((I,\boldsymbol{a}_{I},i)\) with \(\emptyset\ne I\subseteq [r]\), \(\abs{I}\le d+1\), \(\boldsymbol{a}_{I}\in\prod_{j\in I}[s_{j}]\), and \(i\in I\). In the full-tuple alternative, it is the pair
\[
\mathcal{D}=\left(\omega,\left(P\cap H_{i,\boldsymbol{a}}\right)_{\boldsymbol{a}\in\prod_{j=1}^{r}[s_{j}],i\in\omega(\boldsymbol{a})}\right),
\]
where \(\omega(\boldsymbol{a})=I(\boldsymbol{a})\subseteq [r]\) records the Helly witness chosen for the full component choice \(\boldsymbol{a}\).
\begin{lemma}\label{lem:pattern-count}
Fix an \(m\)-point set \(P\subseteq\R^{d}\) and put \(T=T_{d,r}(\boldsymbol{s})\). The tuple \(\mathcal{D}\) has at most \(\exp(T\log(er))\tau_{d}(P)^{T}\) possible values. Moreover, each possible value of \(\mathcal{D}\) determines the subsets \(Q_{i}\cap P\), \(i\in [r]\).
\end{lemma}
\begin{proof}[Proof of Lemma~\ref{lem:pattern-count}]
Recall that \(\tau_{d}(P)\) is the number of subsets of the form \(P\cap H\), where \(H\) is a closed halfspace. We count possible records \(\mathcal{D}\) on the fixed set \(P\). If \(H\) is a closed halfspace, then \(P\cap H\) has at most \(\tau_{d}(P)\) possible values. The same bound also covers the unrestricted set \(\R^{d}\), since \(P\cap\R^{d}=P\), and \(P\) is obtained from a closed halfspace containing all points of \(P\).

First suppose that the sparse-subtuple alternative is selected, so \(T=H_{d,r}(\boldsymbol{s})\). The record has one coordinate \(P\cap K_{i,I,\boldsymbol{a}_{I}}\) for each triple \((I,\boldsymbol{a}_{I},i)\) with \(\emptyset\ne I\subseteq [r]\), \(\abs{I}\le d+1\), \(\boldsymbol{a}_{I}\in\prod_{j\in I}[s_{j}]\), and \(i\in I\). The number of such triples is
\[
\sum_{\substack{\emptyset\ne I\subseteq [r]\\ \abs{I}\le d+1}}\abs{I}\prod_{j\in I}s_{j}=H_{d,r}(\boldsymbol{s})=T.
\]
Each coordinate has at most \(\tau_{d}(P)\) possible values, so this construction gives at most \(\tau_{d}(P)^{T}\) records.

Now suppose that the full-tuple alternative is selected, so \(T=qS(\boldsymbol{s})\). For each full component choice \(\boldsymbol{a}\), the map \(\omega\) chooses a nonempty subset of \([r]\) of size at most \(q\). Hence the number of possible maps \(\omega\) is at most
\[
\left(\sum_{j=1}^{q}\binom{r}{j}\right)^{S(\boldsymbol{s})}\le (er)^{qS(\boldsymbol{s})}=\exp(T\log(er)).
\]
After \(\omega\) is fixed, the remaining coordinates of \(\mathcal{D}\) are the subsets \(P\cap H_{i,\boldsymbol{a}}\) with \(i\in\omega(\boldsymbol{a})\). Their number is \(\sum_{\boldsymbol{a}}\abs{\omega(\boldsymbol{a})}\le qS(\boldsymbol{s})=T\), so these coordinates have at most \(\tau_{d}(P)^{T}\) possible values. Thus this construction gives at most \(\exp(T\log(er))\tau_{d}(P)^{T}\) records. Combining this with the previous case gives the claimed upper bound on the number of possible records \(\mathcal{D}\).

It remains to check that the record determines \(Q_{i}\cap P\). In the sparse-subtuple alternative, for fixed \(i\) and \(a\),
\[
Q_{i,a}\cap P=\bigcap_{\substack{I,\boldsymbol{a}_{I}:i\in I\\ a_{i}=a}}\left(P\cap K_{i,I,\boldsymbol{a}_{I}}\right).
\]
In the full-tuple alternative, for fixed \(i\) and \(a\),
\[
Q_{i,a}\cap P=\bigcap_{\substack{\boldsymbol{b}\in\prod_{j=1}^{r}[s_{j}]:b_{i}=a\\ i\in\omega(\boldsymbol{b})}}\left(P\cap H_{i,\boldsymbol{b}}\right),
\]
where an empty intersection is interpreted as \(P\). Therefore \(\mathcal{D}\) determines every \(Q_{i,a}\cap P\), and then \(Q_{i}\cap P=\bigcup_{a=1}^{s_{i}}(Q_{i,a}\cap P)\).  
\end{proof}

\begin{proof}[Proof of Theorem~\ref{thm:general-upper}]
Recall that \(T=T_{d,r}(\boldsymbol{s})\). Since \(r\ge 2\) and all \(s_{i}\ge 1\), we have \(T\ge 2\). Suppose \(P\subseteq\R^{d}\) has \(m\) points and every surjective map \(P\to [r]\) is bad. We prove \(m\le cdrT\log(erT)\). This is enough, because then any larger point set has at least one surjective map whose parts have the desired intersection property.

After increasing the absolute constant, we may assume \(m\ge d+1\) and \(m\ge r\log(2r)\), otherwise the desired bound is immediate from \(T\ge 2\). There are \(r^{m}\) maps \(P\to [r]\) in total. A map that is not surjective misses at least one value, choosing a missed value and then mapping all points to the remaining \(r-1\) values gives the upper bound \(r(r-1)^{m}\) for the number of non-surjective maps. Hence the number of surjective maps is at least \(r^{m}-r(r-1)^{m}\ge r^{m}(1-re^{-\frac{m}{r}})\ge \frac{r^{m}}{2}\).

Recall that a map \(\chi:P\to [r]\) is bad if, with \(P_{i}=\chi^{-1}(i)\), there are \(s_{i}\)-convex sets \(C_{i}\supseteq P_{i}\) such that \(\bigcap_{i=1}^{r}C_{i}=\emptyset\).
For each bad map, choose witnessing sets \(C_{1},\ldots,C_{r}\). We apply Lemma~\ref{lem:component-reduction} to obtain compact convex sets \(D_{i,a}\), then apply Lemma~\ref{lem:polyhedral-reduction} to obtain the polyhedral sets \(Q_{i}\). We use Lemma~\ref{lem:pattern-count} to bound the possible finite records \(\mathcal{D}\). The next claim fixes one such record and counts how many bad maps can lead to it. This is the key counting step of the proof: the factor \((r-1)^{m}\), rather than \(r^{m}\), is exactly what later gives the improved bound.
\begin{claim}\label{claim:one-tuple}
Fix one possible value of \(\mathcal{D}\). There are at most \((r-1)^{m}\) maps \(\chi:P\to [r]\) that give this same value of \(\mathcal{D}\).
\end{claim}
\begin{poc}
    The fixed value of \(\mathcal{D}\) determines the sets \(Q_{1}\cap P,\ldots,Q_{r}\cap P\). These sets have empty total intersection, because the actual sets \(Q_{1},\ldots,Q_{r}\) built in Lemma \ref{lem:polyhedral-reduction} have empty total intersection. Therefore no point of \(P\) can lie in all \(r\) of the sets \(Q_{i}\).

    Now take a bad map \(\chi:P\to [r]\) that gives this same value of \(\mathcal{D}\). If \(\chi(\boldsymbol{x})=i\), then the construction has \(P_{i}\subseteq Q_{i}\), so in particular \(\boldsymbol{x}\in Q_{i}\). But \(\boldsymbol{x}\) is missing from at least one of the \(r\) sets \(Q_{1},\ldots,Q_{r}\). Hence there are at most \(r-1\) possible values of \(\chi(\boldsymbol{x})\). This is true independently for every point of \(P\), so a fixed value of \(\mathcal{D}\) can come from at most \((r-1)^{m}\) maps.

\end{poc}
There are at least \(\frac{r^{m}}{2}\) surjective maps, and by our assumptions, all of them are bad. We apply Lemma~\ref{lem:pattern-count} to bound the number of possible values of \(\mathcal{D}\), and Claim~\ref{claim:one-tuple} to bound how many maps can lead to one fixed value. Therefore
\[
\frac{1}{2}r^{m}\le (r-1)^{m}\exp(T\log(er))\tau_{d}(P)^{T}.
\]
We apply Lemma~\ref{lem:halfspace-subsets} to obtain \(\tau_{d}(P)\le \left(\frac{em}{d+1}\right)^{d+1}\). Taking logarithms yields
\[
m\log\left(\frac{r}{r-1}\right)\le T\log(er)+(d+1)T\log\left(\frac{em}{d+1}\right)+\log 2.
\]
Since \(\log\left(\frac{r}{r-1}\right)\ge \frac{1}{r}\), \(d+1\le 2d\), \(m\ge d+1\), and \(T\ge 2\), there is an absolute constant \(c_{1}\) such that
\[
\frac{m}{d}\le c_{1}rT\log\left(\frac{c_{1}erm}{d}\right).
\]
We shall use the elementary fact that if \(x\le a\log(bx)\), with
\(a,b\ge 2\), then \(x\le c_{\mathrm{abs}}a\log(ab)\) for an absolute constant \(c_{\mathrm{abs}}\).
Applying this with \(x=\frac{m}{d}\), \(a=c_{1}rT\), and \(b=c_{1}er\), gives 
\(m\le c_{2}drT\log(erT)\) for some absolute constant \(c_{2}>0\). This proves the theorem.

\end{proof}

\subsection{Proof of Theorem~\ref{thm:block}}
We now prove the complete-transversal version. The geometric part of the proof is exactly the one just used for Theorem~\ref{thm:general-upper}: once a labeled partition has been fixed, a bad partition produces the same finite record \(\mathcal{D}\). What changes is only the final counting step. In Theorem~\ref{thm:general-upper}, each point had at most \(r-1\) possible labels after \(\mathcal{D}\) was fixed, here the labels inside each prescribed \(r\)-point class must form a permutation, so we need a permanent bound.

The following elementary consequence of Lemma~\ref{lem:bregman} is the main tool in this proof. It states that, if every point in one prescribed class is missing at least one label, then at most half of all permutations of the labels remain possible.
\begin{lemma}\label{lem:permanent}
Let \(M\) be an \(r\times r\) zero-one matrix whose row sums are all at most \(r-1\). Then \(\operatorname{per}(M)\le \frac{r!}{2}\).
\end{lemma}
\begin{proof}[Proof of Lemma~\ref{lem:permanent}]
Let the row sums be \(a_{1},\ldots,a_{r}\). If some \(a_{j}=0\), then the permanent is zero. Otherwise we apply Lemma~\ref{lem:bregman} and obtain \(\operatorname{per}(M)\le \prod_{j=1}^{r}(a_{j}!)^{\frac{1}{a_{j}}}\). Since each \(a_{j}\le r-1\), we get \(\operatorname{per}(M)\le ((r-1)!)^{\frac{r}{r-1}}\), which is at most \(\frac{r!}{2}\).
\end{proof}

\begin{proof}[Proof of Theorem~\ref{thm:block}]
We again put \(T=T_{d,r}(\boldsymbol{s})\). Let \(A_{1},\ldots,A_{m}\subseteq\R^{d}\) be pairwise disjoint \(r\)-point sets, and put \(P=A_{1}\sqcup\cdots\sqcup A_{m}\). Suppose, toward the contrapositive, that every labeled partition \(P=P_{1}\sqcup\cdots\sqcup P_{r}\) with \(\abs{P_{i}\cap A_{j}}=1\) for all \(i,j\) is bad. We then show \(m\le cdT\log(erT)\).

Also, after increasing the absolute constant, we may assume \(m\ge d+1\), otherwise the desired bound is immediate from \(T\ge 2\). Notice that there are exactly \((r!)^{m}\) labeled partitions of the required complete-transversal form. Each of them is bad. We apply Lemmas~\ref{lem:component-reduction} and~\ref{lem:polyhedral-reduction}, followed by Lemma~\ref{lem:pattern-count} on the ground set \(P\), to assign one finite record \(\mathcal{D}\) to it. Since \(\abs{P}=mr\ge d+1\), we apply Lemma~\ref{lem:halfspace-subsets} and obtain \(\tau_{d}(P)\le \left(\frac{emr}{d+1}\right)^{d+1}\). Hence the number of possible values of \(\mathcal{D}\) is at most \(\exp(T\log(er))\left(\frac{emr}{d+1}\right)^{(d+1)T}\).

Fix one value of \(\mathcal{D}\). It determines the sets \(Q_{1}\cap P,\ldots,Q_{r}\cap P\), and their total intersection is empty. For each point \(\boldsymbol{x}\in P\), define \(L(\boldsymbol{x})=\{i\in [r]:\boldsymbol{x}\in Q_{i}\}\). Then \(\abs{L(\boldsymbol{x})}\le r-1\) for every \(\boldsymbol{x}\in P\).

Consider one prescribed class, say \(A_{j}=\{\boldsymbol{x}_{j,1},\ldots,\boldsymbol{x}_{j,r}\}\). A labeled partition of this class is a permutation \(\sigma\in S_{r}\), where \(\boldsymbol{x}_{j,a}\) is put into part \(\sigma(a)\). This permutation can lead to the fixed value of \(\mathcal{D}\) only if \(\sigma(a)\in L(\boldsymbol{x}_{j,a})\) for every \(a\in [r]\). The number of such permutations is the permanent of the zero-one matrix \(M^{(j)}\) with \(M^{(j)}_{a,i}=1\) exactly when \(i\in L(\boldsymbol{x}_{j,a})\). Every row of \(M^{(j)}\) has at most \(r-1\) ones, so we apply Lemma~\ref{lem:permanent} and obtain at most \(\frac{r!}{2}\) choices for this class. The prescribed classes are independent after \(\mathcal{D}\) is fixed, and therefore one fixed value of \(\mathcal{D}\) can come from at most \((\frac{r!}{2})^{m}\) labeled partitions.

Counting all \((r!)^{m}\) labeled partitions by the value of \(\mathcal{D}\) gives
\[
(r!)^{m}\le \left(\frac{r!}{2}\right)^{m}\exp(T\log(er))\left(\frac{emr}{d+1}\right)^{(d+1)T}.
\]
After canceling \((r!)^{m}\) and taking logarithms, we obtain \(m\log 2\le T\log(er)+(d+1)T\log\left(\frac{emr}{d+1}\right)\). Since \(d+1\le 2d\) and \(m\ge d+1\), there is an absolute constant \(c_{1}\) such that
\[
\frac{m}{d}\le c_{1}T\log\left(\frac{c_{1}erm}{d}\right).
\]
The same logarithmic absorption argument as above gives \(m\le c_{2}dT\log(erT)\). This finishes the proof.
\end{proof}

\section{Proof of Theorem~\ref{thm:rainbow-obstruction}}
\begin{proof}[Proof of Theorem~\ref{thm:rainbow-obstruction}]
Recall that \(\gamma_{s}(r,d)\) is the least integer \(t\) with the following property. For every collection of pairwise disjoint finite sets \(Y_{0},\ldots,Y_{d}\subseteq\R^{d}\) with \(\abs{Y_{j}}\ge t\), there are pairwise disjoint rainbow sets \(R_{1},\ldots,R_{r}\), meaning \(\abs{R_{i}\cap Y_{j}}=1\) for every \(i\in [r]\) and \(j\in\{0,\ldots,d\}\), such that every choice of \(s\)-convex sets \(C_{i}\supseteq R_{i}\), \(i\in [r]\), satisfies \(\bigcap_{i=1}^{r}C_{i}\ne\emptyset\). Fix \(r\ge2\), \(d\ge1\), and \(s\ge2\). It is enough to show that, for every integer \(t\), there are color classes \(Y_{0},\ldots,Y_{d}\) of size \(t\) for which no \(r\) disjoint rainbow sets satisfy this condition. If \(t<r\), this is immediate, because a class of size \(t\) cannot meet \(r\) disjoint rainbow sets. Hence assume \(t\ge r\).

Let \(\boldsymbol{\mu}:\R\to\R^{d}\) be the moment curve, \(\boldsymbol{\mu}(u)=(u,u^{2},\ldots,u^{d})\). Choose \((d+1)t\) distinct points on this curve and distribute \(t\) of them to each \(Y_{j}\). Any at most \(d+1\) of these points are affinely independent. Indeed, for parameters \(u_{0},\ldots,u_{k}\) with \(k\le d\), the relevant affine independence matrix contains the Vandermonde matrix with rows \(1,u,\ldots,u^{k}\), whose determinant is \(\prod_{0\le a<b\le k}(u_{b}-u_{a})\ne0\).

We need the following property of the moment curve. It is automatic in odd dimension, the point of the proof is to handle even dimension as well.

\begin{claim}\label{claim:two-piece-moment}
Let \(A\) and \(B\) be disjoint \((d+1)\)-point subsets of the moment curve in \(\R^{d}\). There are partitions \(A=A_{0}\sqcup A_{1}\) and \(B=B_{0}\sqcup B_{1}\) such that \(\operatorname{conv}(A_{\alpha})\cap\operatorname{conv}(B_{\beta})=\emptyset\) for all \(\alpha,\beta\in\{0,1\}\).
\end{claim}

\begin{poc}
Suppose first that \(d=2k+1\). Split each of \(A\) and \(B\) into two sets of size \(k+1\). For every \(\alpha,\beta\in\{0,1\}\), the set \(A_{\alpha}\sqcup B_{\beta}\) has \(2k+2=d+1\) points and is therefore affinely independent. Its two disjoint subsets have disjoint convex hulls, since an intersection would give two convex combinations of the same point and hence a nonzero affine dependence.

Now let \(d=2k\). We first record the affine dependence among \(d+2\) points on the moment curve. If \(u_{1}<\cdots<u_{d+2}\), put \(\lambda_{i}=\left(\prod_{j\ne i}(u_{i}-u_{j})\right)^{-1}\). The coefficient of \(x^{d+1}\) in the Lagrange interpolation formula for \(x^{m}\) at \(u_{1},\ldots,u_{d+2}\) gives \(\sum_{i=1}^{d+2}\lambda_{i}u_{i}^{m}=0\) for every \(m\in\{0,\ldots,d\}\). Thus \(\sum_{i=1}^{d+2}\lambda_{i}=0\) and \(\sum_{i=1}^{d+2}\lambda_{i}\boldsymbol{\mu}(u_{i})=\boldsymbol{0}\). The signs of \(\lambda_{1},\ldots,\lambda_{d+2}\) alternate. Moreover, this affine dependence is unique up to multiplication by a nonzero scalar, because every \(d+1\) of the points are affinely independent. If a bipartition of these \(d+2\) points had intersecting convex hulls, subtracting the two corresponding convex combinations would give an affine dependence. Uniqueness and the fact that every \(\lambda_{i}\) is nonzero force the two parts to be exactly the two alternating classes in the parameter order.

Write \(A=\{\boldsymbol{\mu}(a_{1}),\ldots,\boldsymbol{\mu}(a_{2k+1})\}\), where \(a_{1}<\cdots<a_{2k+1}\). Among the \(2k\) open intervals \((a_{i},a_{i+1})\), at least one contains at most one parameter belonging to a point of \(B\), otherwise these intervals would contain at least \(4k>2k+1=\abs{B}\) such parameters. Fix such an interval \((a_{h},a_{h+1})\). Choose \(A_{0}\) to contain \(\boldsymbol{\mu}(a_{h})\), \(\boldsymbol{\mu}(a_{h+1})\), and any further \(k-1\) points of \(A\). At least \(2k\) points of \(B\) have parameters outside this interval, so choose \(B_{0}\) to be any \(k+1\) of them. Set \(A_{1}=A\setminus A_{0}\) and \(B_{1}=B\setminus B_{0}\).

The two points \(\boldsymbol{\mu}(a_{h})\) and \(\boldsymbol{\mu}(a_{h+1})\) are consecutive in the parameter order on \(A_{0}\sqcup B_{0}\), so the bipartition \(A_{0}\sqcup B_{0}\) is not alternating. By the preceding paragraph, \(\operatorname{conv}(A_{0})\cap\operatorname{conv}(B_{0})=\emptyset\). For every other pair \((\alpha,\beta)\ne(0,0)\), at least one of \(A_{\alpha}\) and \(B_{\beta}\) has size \(k\). Their union therefore has at most \(2k+1=d+1\) points and is affinely independent, which again implies that their convex hulls are disjoint. This proves the claim.
\end{poc}

Now fix arbitrary pairwise disjoint rainbow sets \(R_{1},\ldots,R_{r}\). Each \(R_{i}\) has exactly \(d+1\) points. Apply Claim \ref{claim:two-piece-moment} to \(R_{1}\) and \(R_{2}\), obtaining partitions \(R_{1}=R_{1,0}\sqcup R_{1,1}\) and \(R_{2}=R_{2,0}\sqcup R_{2,1}\). Set \(C_{1}=\operatorname{conv}(R_{1,0})\cup\operatorname{conv}(R_{1,1})\) and \(C_{2}=\operatorname{conv}(R_{2,0})\cup\operatorname{conv}(R_{2,1})\). These sets are \(s\)-convex because \(s\ge2\), and Claim \ref{claim:two-piece-moment} gives \(C_{1}\cap C_{2}=\emptyset\). Taking \(C_{i}=\R^{d}\) for \(i=3,\ldots,r\), we obtain \(\bigcap_{i=1}^{r}C_{i}=\emptyset\).

Thus every possible choice of pairwise disjoint rainbow sets fails the robust condition for these classes \(Y_{0},\ldots,Y_{d}\). Since the construction works for every \(t\), no finite value of \(\gamma_{s}(r,d)\) exists. This finishes the proof.
\end{proof}

\section*{Acknowledgment}
The authors are grateful to Prof. Gil Kalai for many insightful comments and stimulating discussions following our previous joint work with Wenchong Chen and Zhouningxin Wang on this topic~\cite{ChenGeShuWangXu2025}. In particular, he suggested investigating whether the dependence on \(r\) in the Alon--Smorodinsky upper bound could be improved, and whether analogous colored extensions could be established for unions of convex sets. The authors also thank Prof. Noga Alon and Prof. Shakhar Smorodinsky for their inspiring comments and helpful discussions on this topic.

\bibliographystyle{abbrv}
\bibliography{tverberg_topics}
\end{document}